\documentclass[12pt]{article}
\usepackage{geometry}
\usepackage{array}
\usepackage{amsmath}
\usepackage{amsthm}
\usepackage{amssymb}
\usepackage{microtype}
\usepackage{longtable}
\usepackage{multirow}

\makeatletter

\providecommand{\tabularnewline}{\\}
\providecommand{\qbinom}[2]{\binom{#1}{#2}_{\!\!q}}

\numberwithin{equation}{section}
\numberwithin{figure}{section}
\theoremstyle{plain}
\newtheorem{thm}{\protect\theoremname}[section]
  \theoremstyle{plain}
  \newtheorem{lem}[thm]{\protect\lemmaname}
  \theoremstyle{plain}
  \newtheorem{prop}[thm]{\protect\propositionname}
  \theoremstyle{plain}
  \newtheorem{cor}[thm]{\protect\corollaryname}
  \theoremstyle{plain}
  \newtheorem{conjecture}[thm]{\protect\conjecturename}
  \theoremstyle{plain}
  \newtheorem{question}[thm]{\protect\questionname}

\usepackage{mathtools}
\usepackage{amsmath, amscd}
\usepackage[pdfencoding=auto]{hyperref}

\DeclareMathOperator{\Pk}{Pk}
\DeclareMathOperator{\Lpk}{Lpk}
\DeclareMathOperator{\Rpk}{Rpk}
\DeclareMathOperator{\Epk}{Epk}
\DeclareMathOperator{\Val}{Val}
\DeclareMathOperator{\pk}{pk}
\DeclareMathOperator{\val}{val}
\DeclareMathOperator{\lpk}{lpk}
\DeclareMathOperator{\rpk}{rpk}
\DeclareMathOperator{\epk}{epk}
\DeclareMathOperator{\br}{br}
\DeclareMathOperator{\udr}{udr}
\DeclareMathOperator{\des}{des}
\DeclareMathOperator{\asc}{asc}
\DeclareMathOperator{\maj}{maj}
\DeclareMathOperator{\comaj}{comaj}
\DeclareMathOperator{\st}{st}
\DeclareMathOperator{\Des}{Des}

\DeclareMathOperator{\lfr}{lfr}
\DeclareMathOperator{\sfr}{sfr}
\DeclareMathOperator{\lir}{lir}
\DeclareMathOperator{\sir}{sir}
\DeclareMathOperator{\lr}{lr}
\DeclareMathOperator{\id}{id}
\DeclareMathOperator{\Comp}{Comp}
\DeclareMathOperator{\Span}{Span}
\DeclareMathOperator{\QSym}{QSym}

\usepackage{titlesec}
\titleformat{\section}
  {\normalfont\large\bfseries}{\thesection.}{.8ex}{} 
\titleformat{\subsection}
  {\normalfont\bfseries}{\thesubsection.}{.8ex}{}

\let\originalleft\left
\let\originalright\right
\renewcommand{\left}{\mathopen{}\mathclose\bgroup\originalleft}
\renewcommand{\right}{\aftergroup\egroup\originalright}

\renewcommand{\mid}{:}

\makeatother

\providecommand{\conjecturename}{Conjecture}
\providecommand{\corollaryname}{Corollary}
\providecommand{\lemmaname}{Lemma}
\providecommand{\propositionname}{Proposition}
\providecommand{\questionname}{Question}
\providecommand{\theoremname}{Theorem}

\newcommand\pair[2]{\left\langle#1,#2\right\rangle}
\newcommand{\h}{\mathbf{h}}
\renewcommand{\r}{\mathbf{r}}
\newcommand{\Sym}{\mathbf{Sym}}
\newcommand{\Symxy}{\Sym_{xy}}
\newcommand{\Symtxy}{\Sym_{txy}}
\newcommand{\ncsf}{noncommutative symmetric function}
\newcommand{\ncsfs}{\ncsf s}
\newcommand{\Q}{\mathbb{Q}}
\newcommand{\N}{\mathbb{N}}
\newcommand{\Qtxy}{\Q[[t*]][x,y]}
\newcommand{\Qtx}{\Q[[t*]][x]}
\newcommand{\ceil}[1]{\left\lceil#1\right\rceil}
\newcommand{\floor}[1]{\left\lfloor#1\right\rfloor}
\newcommand{\px}{p}
\newcommand{\qx}{q}

\begin{document}

\title{Shuffle-compatible permutation statistics}

\author{Ira M. Gessel\qquad{}Yan Zhuang\\
Department of Mathematics\\
Brandeis University%
\thanks{MS 050, Waltham, MA 02453
}\\
\texttt{\{gessel, zhuangy\}@brandeis.edu}}
\maketitle
\begin{abstract}
Since the early work of Richard Stanley, it has been observed that several permutation statistics have a remarkable property with respect to shuffles of permutations. We formalize this notion of a shuffle-compatible permutation statistic and introduce the shuffle algebra of a shuffle-compatible permutation statistic, which encodes the distribution of the statistic over shuffles of permutations. This paper develops a theory of shuffle-compatibility for descent statistics---statistics that depend only on the descent set and length---which has close connections to the theory of $P$-partitions, quasisymmetric functions, and noncommutative symmetric functions. We use our framework to prove that many descent statistics are shuffle-compatible and to give explicit descriptions of their shuffle algebras, thus unifying past results of Stanley, Gessel, Stembridge, Aguiar--Bergeron--Nyman, and Petersen.
\end{abstract}
\textbf{\small{}Keywords:}{\small{} permutations, shuffles, permutation statistics, $P$-partitions, quasisymmetric functions, noncommutative
symmetric functions}{\small \par}
{\let\thefootnote\relax\footnotetext{The first author was supported by a grant from the Simons Foundation (\#427060, Ira Gessel).}}
{\let\thefootnote\relax}

\tableofcontents{}

\section{Introduction}

We say that $\pi=\pi_{1}\pi_{2}\cdots\pi_{n}$ is a \textit{permutation}
of \textit{length} $n$ (or an $n$-\textit{permutation}) if it is
a sequence of $n$ distinct letters---not necessarily from 1 to $n$---in
$\mathbb{P}$, the set of positive integers.
For example, $\pi=47381$
is a permutation of length 5. 
Let $\left|\pi\right|$ denote the length of a permutation $\pi$ and
let $\mathfrak{P}_{n}$ denote the set
of all permutations of length $n$.\footnote{In Section \ref{s-section2}, we will in a few instances consider permutations with a letter 0. We note that, in these cases, every property of permutations that is used still holds when 0 is allowed to be a letter.
}

A \textit{permutation statistic} (or \textit{statistic}) $\st$ is
a function defined on permutations such that $\st(\pi)=\st(\sigma)$
whenever $\pi$ and $\sigma$ are permutations with the same relative
order.\footnote{Define the \textit{standardization} of an $n$-permutation $\pi$ to be the permutation of $[n]$ obtained by  replacing the $i$th smallest letter of $\pi$ with $i$ for  $i$ from 1 to $n$. Then two permutations are said to \textit{have the same relative order} if they have the same standardization.} Three classical examples of permutation statistics are the
descent set $\Des$, the descent number $\des$, and the major index
$\maj$. We say that $i\in[n-1]$ is a \textit{descent} of $\pi\in\mathfrak{P}_{n}$
if $\pi_{i}>\pi_{i+1}$. Then the \textit{descent set} 
\[
\Des(\pi)\coloneqq\{\, i\in[n-1]\mid\pi_{i}>\pi_{i+1}\,\}
\]
of $\pi$ is the set of its descents, the \textit{descent number}
\[
\des(\pi)\coloneqq\left|\Des(\pi)\right|
\]
its number of descents, and the \textit{major index} 
\[
\maj(\pi)\coloneqq\sum_{k\in\Des(\pi)}k
\]
the sum of its descents.

Let $\pi\in\mathfrak{P}_{m}$ and $\sigma\in\mathfrak{P}_{n}$ be
\textit{disjoint} permutations, that is, permutations with no letters
in common. We say that $\tau\in\mathfrak{P}_{m+n}$ is a \textit{shuffle}
of $\pi$ and $\sigma$ if both $\pi$ and $\sigma$ are subsequences
of $\tau$. The set of shuffles of $\pi$ and $\sigma$ is denoted
$S(\pi,\sigma)$. For example, $S(53,16)=\{5316,5136,5163,1653,1536,1563\}$.
It is easy to see that the number of permutations in $S(\pi,\sigma)$
is ${m+n \choose m}$.

Richard Stanley's theory of $P$-partitions \cite{Stanley1972} implies that the descent
set statistic has a remarkable property related to shuffles: for any disjoint
permutations $\pi$ and $\sigma$, the multiset $\{\,\Des(\tau)\mid\tau\in S(\pi,\sigma)\,\}$---which encodes the distribution of the descent set over shuffles of $\pi$ and $\sigma$---depends only on $\Des(\pi)$, $\Des(\sigma)$, and the lengths of
$\pi$ and $\sigma$ \cite[Exercise 3.161]{Stanley2011}. That is,
if $\pi$ and $\pi^{\prime}$ are permutations of the same length
with the same descent set, and similarly with $\sigma$ and $\sigma^{\prime}$,
then the number of permutations in $S(\pi,\sigma)$ with any given
descent set is the same as the number of permutations in $S(\pi^{\prime},\sigma^{\prime})$
with that descent set.

Stanley also proved a similar but more refined result for the joint
statistic $(\des,\maj)$, which is a special case of \cite[Proposition 12.6 (ii)]{Stanley1972}.
Bijective proofs were later found by Goulden \cite{Goulden1985} and
by Stadler \cite{Stadler1999}; they referred to this result as ``Stanley's
shuffling theorem''. Recall that the $q$-binomial coefficient ${n \choose k}_{q}$
is defined by 
\[
\qbinom{n}{k}\coloneqq\frac{[n]_{q}!}{[k]_{q}!\,[n-k]_{q}!}
\]
where $[n]_{q}!\coloneqq(1+q)(1+q+q^{2})\cdots(1+q+\cdots+q^{n-1})$.
\begin{thm}[Stanley's shuffling theorem]

Let $\pi\in\mathfrak{P}_{m}$ and $\sigma\in\mathfrak{P}_{n}$ be
disjoint permutations, and let $S_{k}(\pi,\sigma)$ be the set of
shuffles of $\pi$ and $\sigma$ with exactly $k$ descents. Then
\begin{multline}
\qquad\sum_{\tau\in S_{k}(\pi,\sigma)}q^{\maj(\tau)}=q^{\maj(\pi)+\maj(\sigma)+(k-\des(\pi))(k-\des(\sigma))}\\
\times\qbinom{m-\des(\pi)+\des(\sigma)}{k-\des(\pi)}\qbinom{n-\des(\sigma)+\des(\pi)}{ k-\des(\sigma)}.
\qquad
\label{e-qshuffle}
\end{multline}

\end{thm}
A variant of the theorem gives the formula

\begin{equation}
\sum_{\tau\in S(\pi,\sigma)}q^{\maj(\tau)}=q^{\maj(\pi)+\maj(\sigma)}\qbinom{m+n}{m};\label{e-maj}
\end{equation}
see \cite[p.~43]{Stanley1972}. These formulas show that the statistics
$(\des,\maj)$ and $\maj$ have the same property as $\Des$, and
setting $q=1$ in \eqref{e-qshuffle} shows that $\des$ has this
property as well.

We call this property ``shuffle-compatibility''. More precisely, we
say that a permutation statistic $\st$ is \textit{shuffle-compatible}
if for any disjoint permutations $\pi$ and $\sigma$, the multiset $\{\,\st(\tau)\mid\tau\in S(\pi,\sigma)\,\}$
depends only on $\st(\pi)$, $\st(\sigma)$, $\left|\pi\right|$,
and $\left|\sigma\right|$. Hence $\Des$, $\des$, $\maj$, and
$(\des,\maj)$ are examples of shuffle-compatible permutation statistics.

This paper serves as the first in-depth investigation of shuffle-compatibility,
and we focus on the shuffle-compatibility of descent statistics, which are statistics that
depend only on the descent set and length of a permutation. All of
the statistics mentioned so far are descent statistics. In Section
\ref{s-section2}, we introduce some aspects of the general theory of descents in
permutations and define some other descent statistics that we will
be studying in this paper, including the peak set $\Pk$, the peak
number $\pk$, the left peak set $\Lpk$, the left peak number $\lpk$, and the number of up-down
runs $\udr$. There, we also give a bijective proof of the shuffle-compatibility of the descent set.

In Section \ref{s-section3}, we define the ``shuffle algebra'' of a shuffle-compatible
permutation statistic $\st$, which has a natural basis whose structure constants encode the distribution of $\st$ over shuffles of permutations (or more precisely, equivalence classes of permutations induced by the statistic $\st$). Our first result is a characterization of the major index shuffle algebra using the variant (\ref{e-maj}) of Stanley's shuffling theorem. We then prove several basic results that relate the shuffle algebras of permutation statistics that are related in various ways. Notably, if two statistics are related by a basic symmetry---reversion,
complementation, or reverse complementation---and one of them is known
to be shuffle-compatible, then both statistics are shuffle-compatible
and have isomorphic shuffle algebras.

In Section \ref{s-section4}, we introduce the algebra of quasisymmetric functions
$\QSym$ (originally studied in \cite{Gessel1984}) and observe that
it is isomorphic to the descent set shuffle algebra. We establish
a necessary and sufficient condition for the shuffle-compatibility
of a descent statistic, which  shows that the shuffle algebra
of any shuffle-compatible descent statistic is isomorphic to a quotient
algebra of $\QSym$. Using this condition, we give explicit descriptions for the shuffle algebras of 
$\des$ and $(\des,\maj)$. We then observe that the peak
set shuffle algebra is isomorphic to Stembridge's ``algebra of peaks''
arising from his study of enriched $P$-partitions \cite{Stembridge1997}---thus showing
that the peak set $\Pk$ is shuffle-compatible---and use Stembridge's
peak quasisymmetric functions to characterize the peak number shuffle
algebra, thus showing that the peak number $\pk$ is shuffle-compatible. In the same vein, Petersen's work  \cite{Petersen2006,Petersen2007} on left enriched
$P$-partitions implies that the left peak set $\Lpk$ and left peak
number $\lpk$ are shuffle-compatible.

In Section \ref{s-section5}, we introduce the bialgebra of noncommutative symmetric
functions $\mathbf{Sym}$ (originally studied in \cite{ncsf1}), whose
coalgebra structure is dual to the algebra structure of $\QSym$.
By exploiting this duality, we obtain a dual version of our shuffle-compatibility
condition, which allows us to prove shuffle-compatibility of a descent
statistic by constructing a suitable subcoalgebra of $\mathbf{Sym}$.
We use this approach to describe the shuffle algebras of $(\pk,\des)$,
$(\lpk,\des)$, $\udr$, and $(\udr,\des)$, thus showing that these statistics are all
shuffle-compatible. 

Finally, in Section \ref{s-section6}, we provide proofs for an alternate characterization of the $\pk$ and $(\pk,\des)$ shuffle algebras, list some non-shuffle-compatible permutation statistics, and discuss some open questions and conjectures on the topic of shuffle-compatibility.

The appendix of this paper contains two tables. Table 1 lists all permutation statistics that we know to be shuffle-compatible, and Table 2 lists various equivalences (as defined in Section \ref{s-section3}) among the statistics that are studied in this paper.

We note that some permutation statistics, such as the number of inversions, satisfy a weak form of shuffle-compatibility:  
for disjoint permutations $\pi$ and $\sigma$, if every letter of $\pi$ is less than every letter of $\sigma$, then the multiset 
$\{\,\st(\tau)\mid\tau\in S(\pi,\sigma)\,\}$
depends only on $\st(\pi)$, $\st(\sigma)$, $\left|\pi\right|$,
and $\left|\sigma\right|$.
Permutation statistics with this property
are associated with
quotients of the Malvenuto--Reutenauer algebra (also called the algebra
of free quasisymmetric functions).  
Some of these statistics have been studied by Vong \cite{Vong2013}, but we do not consider them here.

Also, there is another class of algebras that are related to permutations and their descent sets, based on ordinary multiplication of permutations rather than shuffles. If $\st$  is a function defined on the $n$th symmetric group $\mathfrak{S}_n$, we may consider the elements 
\begin{equation*}
K_\alpha \coloneqq \sum_{\substack{\pi\in \mathfrak{S}_n\\ \st(\pi) = \alpha}}\pi
\end{equation*}
in the group algebra of $\mathfrak{S}_n$, where $\alpha$ ranges over the image of $\st$. Louis Solomon \cite{Solomon1976} proved that if $\st$ is the descent set, then the $K_\alpha$ span a subalgebra of the group algebra of $\mathfrak{S}_n$, called the \emph{descent algebra} of $\mathfrak{S}_n$. Several other descent statistics give subalgebras of the descent algebra, including the descent number \cite{Loday1989}; the peak set \cite{Nyman2003, Schocker2005}; the left peak set, peak number, and left peak number \cite{Aguiar2004, Petersen2006, Petersen2007}; and the number of biruns and up-down runs \cite{Doyle2008, Josuat-Verges2016}. 
These descent statistics have the property that given values $\alpha$ and $\beta$ of $\st$, and $\tau\in \mathfrak{S}_n$, the number of pairs $(\pi,\sigma)$ of permutations in $\mathfrak{S}_n$ with $\st(\pi)=\alpha$, $\st(\sigma)=\beta$, and  $\pi\sigma=\tau$ depends only on $\st(\tau)$. In other words, these statistics are ``compatible'' under the ordinary product of permutations, and our work is an analogue of Solomon's descent theory for statistics compatible under the shuffle product.

Although there is a significant overlap between shuffle-compatible permutation statistics and statistics corresponding to subalgebras of the descent algebra, neither class is contained in the other, as the number of biruns is not shuffle-compatible and the pair $(\pk,\des)$ does not give a subalgebra of the descent algebra. 
The descent algebra and its subalgebras may also be studied through noncommutative symmetric functions (using the internal product of $\Sym$ \cite[Section 5]{ncsf1}) or quasisymmetric functions (using the internal coproduct of $\QSym$ \cite{Gessel1984}).

\section{Permutations and descents}
\label{s-section2}

\subsection{Increasing runs and descent compositions}

We begin with a brief exposition on some basic material in permutation enumeration relating to descents.

Every permutation can be uniquely decomposed into a sequence of maximal
increasing consecutive subsequences, which we call \textit{increasing
runs} (or simply \textit{runs}). For example, the increasing runs
of $21479536$ are $2$, $1479$, $5$, and $36$. Equivalently, an
increasing run of $\pi$ is a maximal consecutive subsequence containing
no descents. Let us call an increasing run \textit{short} if it has
length 1, and \textit{long} if it has length at least 2. The \textit{initial
run} of a permutation refers to its first increasing run, whereas
the \textit{final run} refers to its last increasing run. For example,
the initial run of $21479536$ is $2$ and its final run is $36$.
(If a permutation has only one increasing run, then it is considered to be both an
initial run and a final run.)

The number of increasing runs of a nonempty permutation is one more
than its number of descents; in fact, the lengths of the increasing
runs determine the descents, and vice versa. Given a subset $A\subseteq[n-1]$
with elements $a_{1}<a_{2}<\cdots<a_{j}$, let $\Comp(A)$ be the
composition $(a_{1},a_{2}-a_{1},\dots,a_{j}-a_{j-1},n-a_{j})$ of
$n$, and given a composition $L=(L_{1},L_{2},\dots,L_{k})$ of $n$, let
$\Des(L)\coloneqq\{L_{1},L_{1}+L_{2},\dots,L_{1}+\cdots+L_{k-1}\}$
be the corresponding subset of $[n-1]$. Then, $\Comp$ and $\Des$
are inverse bijections. If $\pi$ is an $n$-permutation with descent
set $A\subseteq[n-1]$, then we call $\Comp(A)$ the \textit{descent
composition} of $\pi$, which we also denote by $\Comp(\pi)$. By
convention, let us say that the empty permutation (i.e., permutation
of length 0) has descent composition $\varnothing$. Note that the
descent composition of $\pi$ gives the lengths of the increasing
runs of $\pi$. Conversely, if $\pi$ has descent composition $L$,
then its descent set $\Des(\pi)$ is $\Des(L)$.

A permutation statistic $\st$ is called a \textit{descent statistic}
if it depends only on the descent composition, that is, if $\Comp(\pi)=\Comp(\sigma)$
implies $\st(\pi)=\st(\sigma)$ for any two permutations $\pi$ and
$\sigma$. Equivalently, $\st$ is a descent statistic if it depends
only on the descent set and length of a permutation. Since two permutations
with the same descent composition must have the same value of $\st$
if $\st$ is a descent statistic, we shall use the notation $\st(L)$
to indicate the value of a descent statistic $\st$ on any permutation
with descent composition $L$.

We define several statistics based on increasing runs: the long run
$\lr$, long initial run $\lir$, long final run $\lfr$, short initial
run $\sir$, and short final run $\sfr$ statistics. Let $\lr(\pi)$
be the number of long runs of $\pi$, let $\lir(\pi)$ be 1 if the
initial run of $\pi$ is long and 0 otherwise, and let $\lfr(\pi)$
be 1 if the final run of $\pi$ is long and 0 otherwise. Also, for
nonempty $\pi$, let $\sir(\pi)\coloneqq1-\lir(\pi)$ and $\sfr(\pi)\coloneqq1-\lfr(\pi)$.
By convention, if $\pi$ is empty, then all of these statistics are
equal to zero. We will use these run statistics to give an alternative
way of characterizing some of the descent statistics introduced in
the next section.

\subsection{Descent statistics}

In the introduction to this paper, we saw four examples of descent
statistics: the descent set $\Des$, descent number $\des$, major
index $\maj$, and the joint statistic $(\des,\maj)$. The following
are additional descent statistics that we will consider in our investigation
of shuffle-compatibility:
\begin{itemize}

\item The comajor index $\comaj$. The \textit{comajor index} $\comaj(\pi)$
of $\pi\in\mathfrak{P}_{n}$, a variant of the major index, is defined
to be 
\[
\comaj(\pi)\coloneqq\sum_{k\in\Des(\pi)}(n-k).
\]

\item The peak set $\Pk$ and peak number $\pk$. We say that $i$ (where $2\leq i\leq n-1$)
is a \textit{peak} of $\pi\in\mathfrak{P}_{n}$ if $\pi_{i-1}<\pi_{i}>\pi_{i+1}$.
The \textit{peak set} $\Pk(\pi)$ of $\pi$ is defined to be 
\[
\Pk(\pi)\coloneqq\{\,2\leq i\leq n-1\mid\pi_{i-1}<\pi_{i}>\pi_{i+1}\,\}
\]
and the \textit{peak number} $\pk(\pi)$ of $\pi$ to be 
\[
\pk(\pi)\coloneqq\left|\Pk(\pi)\right|.
\]

\item The valley set $\Val$ and valley number $\val$. We say that $i$ (where $2\leq i\leq n-1$) is a \textit{valley} of $\pi\in\mathfrak{P}_{n}$ if $\pi_{i-1}>\pi_{i}<\pi_{i+1}$.
Then $\Val(\pi)$ and $\val(\pi)$ are defined in the analogous way.
\item The left peak set $\Lpk$ and left peak number $\lpk$. We say that
$i\in[n-1]$ is a \textit{left peak} of $\pi\in\mathfrak{P}_{n}$
if $i$ is a peak of $\pi$ or if $i=1$ and is a descent of $\pi$. Thus, left peaks of $\pi$ are  peaks of $0\pi$ shifted by 1.
The \textit{left peak set} $\Lpk(\pi)$ is the set of left peaks of $\pi$ and the \textit{left peak number} $\lpk(\pi)$ is the number of left peaks of $\pi$.

\item The right peak set $\Rpk$ and right peak number $\rpk$. These are defined in the same way as the corresponding left peak statistics, except that right peaks of $\pi$ are peaks of $\pi 0$.

\item The exterior peak set $\Epk$ and exterior peak number $\epk$. The \textit{exterior peak set} $\Epk(\pi)$ of $\pi$ is defined by
\[
\Epk(\pi)\coloneqq\begin{cases}
\Lpk(\pi)\cup\Rpk(\pi), & \mathrm{if}\text{ }\left|\pi\right|\neq1\\
\{1\}, & \mathrm{if}\text{ }\left|\pi\right|=1
\end{cases}
\]
and the \textit{exterior peak number} $\epk(\pi)$ of $\pi$ is defined by 
\[
\epk(\pi)\coloneqq\left|\Epk(\pi)\right|.
\]

\item The number of biruns $\br$ and the number of up-down runs $\udr$.
A \textit{birun} of a permutation is a maximal monotone consecutive
subsequence, and the number of biruns of $\pi$ is denoted $\br(\pi)$.
An \textit{up-down run} of a permutation $\pi$
is either a birun or $\pi_{1}$ when $\pi_{1}>\pi_{2}$, and the number
of up-down runs of $\pi$ is denoted $\udr(\pi)$. Thus the up-down runs of $\pi$ are essentially the biruns of $0\pi$. For example, the
biruns of $\pi=871542$ are $871$, $15$, and $542$, and the up-down
runs of $\pi$ are these biruns along with 8, so $\br(\pi)=3$ and
$\udr(\pi)=4$.
\item Ordered tuples of descent statistics, such as $(\pk,\des)$, $(\lpk,\des)$,
and so on.
\end{itemize}

Before continuing, we give two lemmas that will help us
understand some of the above statistics. The first lemma characterizes several statistics in terms of the run
statistics introduced at the end of the previous section, and
the second lemma reveals a close connection between the $\udr$ statistic
and the $\lpk$ and $\val$ statistics.
\begin{lem}
\label{l-pkvalruns} Let $\pi\in\mathfrak{P}_n$  with $n\geq1$. Then
\begin{enumerate}
\item [\normalfont{(a)}] $\pk(\pi)=\lr(\pi)-\lfr(\pi)$
\item [\normalfont{(b)}] $\val(\pi)=\lr(\pi)-\lir(\pi)$
\item [\normalfont{(c)}] 
$\lpk(\pi) = \begin{cases}
\lr(\pi)+\sir(\pi)-\lfr(\pi), & \mbox{if }n\geq2,\\
0, & \mbox{otherwise}.
\end{cases}$
\item [\normalfont{(d)}] $\rpk(\pi)=\lr(\pi)$
\item [\normalfont{(e)}] $\epk(\pi)=\val(\pi)+1$
\end{enumerate}

\end{lem}
\begin{proof}
Part (a) follows from the fact that every non-final long run ends in a peak, and every peak is at the end of a non-final long run. The same is true for valleys and non-initial long runs, and for right peaks and long runs, thus implying (b) and (d).
Next,
\begin{equation*}
\lpk(\pi) =  \begin{cases}
\pk(\pi)+\sir(\pi), & \mbox{if }n\geq2,\\
0, & \mbox{otherwise},
\end{cases}
\end{equation*}
which together with (a) proves (c).
Finally,
\begin{align*}
\epk(\pi) & =\rpk(\pi) + \sir(\pi)\\
 & =\lr(\pi)+1-\lir(\pi)\\
 & =\val(\pi)+1
\end{align*}
proves (e).
\end{proof}
\begin{lem}
\label{l-udr} Let $\pi\in\mathfrak{P}_{n}$ with $n\geq1$. Then
\begin{enumerate}
\item [\normalfont{(a)}]$\udr(\pi)=\lpk(\pi)+\val(\pi)+1$
\item [\normalfont{(b)}]$\lpk(\pi)=\left\lfloor \udr(\pi)/2\right\rfloor $
\item [\normalfont{(c)}]$\val(\pi)=\left\lfloor (\udr(\pi)-1)/2\right\rfloor $
\item [\normalfont{(d)}]If $n\geq2$ and the final run of $\pi$ is short, then $\lpk(\pi)=\val(\pi)+1$.
Otherwise, $\lpk(\pi)=\val(\pi)$.
\end{enumerate}
\end{lem}

This is Lemma 2.1 of \cite{Zhuang2016a}; a proof can be found there. According to this result, not only do $\lpk$ and $\val$ determine $\udr$,
but $\udr$ determines both $\lpk$ and $\val$. In other words, $\udr$
and $(\lpk,\val)$ are equivalent permutation statistics in the sense
that will be formally defined in Section 3.1.

We note that the definitions and properties of descents, increasing runs, descent
compositions, and descent statistics extend naturally to words on
any totally ordered alphabet such as $[n]$ or $\mathbb{P}$ if we replace the strict inequality $<$ with the
weak inequality $\leq$, which reflects the fact that increasing runs
are allowed to be weakly increasing in this setting. For example,
$i$ is a peak of the word $w=w_{1}w_{2}\cdots w_{n}$ if $w_{i-1}\leq w_{i}>w_{i+1}$.

\subsection{Possible values of some descent statistics}

In our study of shuffle-compatibility, it will be useful to determine
all possible values that a descent statistic can achieve. It is clear
that for $\pi\in\mathfrak{P}_{n}$ and $n\geq1$, we have $0\leq\des(\pi)\leq n-1$
and $\des(\pi)$ can attain any value in this range for some $\pi\in\mathfrak{P}_{n}$.
It is also easy to check that the possible values of $\maj(\pi)$
and $\comaj(\pi)$ for $\pi\in\mathfrak{P}_{n}$ range from 0 to ${n \choose 2}$,
and that all of these values are attainable. Finding such bounds for
other descent statistics requires more work. Here, we determine all
possible values for the $(\des,\maj)$, $(\des,\comaj)$, 
$(\pk,\des)$, $(\lpk,\des)$, and $(\udr,\des)$ statistics. 
\goodbreak
\begin{prop}[Possible values of $(\des,\maj)$]
\leavevmode
\begin{enumerate}
\item [\normalfont{(a)}] For any permutation $\pi\in\mathfrak{P}_{n}$
with $n\geq1$ and $\des(\pi)=j$, we have ${j+1 \choose 2}\leq\maj(\pi)\leq nj-{j+1 \choose 2}$.
\item [\normalfont{(b)}] If $n\geq1$, $0\leq j\leq n-1$, and ${j+1 \choose 2}\leq k\leq nj-{j+1 \choose 2}$,
then there exists $\pi\in\mathfrak{P}_{n}$ with $\des(\pi)=j$ and $\maj(\pi)=k$.
\end{enumerate}
\end{prop}
\begin{proof}
Among all $n$-permutations with $j$ descents, it is clear that the
smallest possible value of $\maj$ is attained when the descent set
is $\{1,2,\dots,j\}$, in which case the major index is equal to ${j+1 \choose 2}$.
Similarly, the largest possible value of $\maj$ is attained when
the descent set is $\{n-j,n-j+1,\dots,n-1\}$, in which case the major
index is equal to $nj-{j+1 \choose 2}$. This proves (a).

Next we prove (b). The case $j=0$ is easy, so we assume that $j\ge1$. 
A permutation in $\mathfrak{P}_{n}$ with descent set $\{1,2,\dots,j\}$
has major index ${j+1 \choose 2}$.
Now let $\pi$ be a permutation in $\mathfrak{P}_n$ with $j$ descents, and suppose that for some $i\in\Des(\pi)$ we have $i\ne n-1$ and $i+1\notin\Des(\pi)$.
Take
$\sigma\in\mathfrak{P}_{n}$ to have descent set $(\Des(\pi)\setminus\{i\})\cup\{i+1\}$.
(This is possible because for any subset $A$ of $[n-1]$, there exists
an $n$-permutation with descent set $A$.) Then $\maj(\sigma)=\maj(\pi)+1$.
We can repeat this process to increase the major index by 1 with every
iteration until we reach a permutation with descent set $\{n-j,n-j+1,\dots,n-1\}$, and thus major index $nj-{j+1 \choose 2}$.
This proves (b).
\end{proof}
\begin{prop}[Possible values of $(\des,\comaj)$]
\label{p-comajvalues}
\leavevmode
\begin{enumerate}
\item [\normalfont{(a)}] For any permutation $\pi\in\mathfrak{P}_{n}$
with $n\geq1$ and $\des(\pi)=j$, we have ${j+1 \choose 2}\leq\comaj(\pi)\leq nj-{j+1 \choose 2}$.
\item [\normalfont{(b)}] If $n\geq1$, $0\leq j\leq n-1$, and ${j+1 \choose 2}\leq k\leq nj-{j+1 \choose 2}$,
there exists $\pi\in\mathfrak{P}_{n}$ with $\des(\pi)=j$ and $\comaj(\pi)=k$.
\end{enumerate}
\end{prop}
\begin{proof}
This follows from the previous proposition and the formula $\comaj(\pi)=n\des(\pi)-\maj(\pi)$.\end{proof}
\begin{prop}[Possible values of $(\pk,\des)$]
\label{p-pkdesvalues}
\leavevmode
\begin{enumerate}
\item [\normalfont{(a)}] For any permutation $\pi\in\mathfrak{P}_{n}$
with $n\geq1$, we have $0\leq\pk(\pi)\leq\left\lfloor (n-1)/2\right\rfloor $.
In addition, $\pk(\pi) \leq \des(\pi) \leq n-\pk(\pi) -1$.

\item [\normalfont{(b)}] If $n\geq1$, $0\leq j\leq\left\lfloor (n-1)/2\right\rfloor $,
and $j\leq k\leq n-j-1$, then there exists $\pi\in\mathfrak{P}_{n}$ with
$\pk(\pi)=j$ and $\des(\pi)=k$.
\end{enumerate}
\end{prop}
\begin{proof}
Fix $n\geq1$. Recall from Lemma \ref{l-pkvalruns} (a) that $\pk(\pi)$
is equal to the number of non-final long runs of $\pi$. It is clear
that the number of non-final long runs of an $n$-permutation is
between $0$ and $\left\lfloor (n-1)/2\right\rfloor $. Every peak is a descent, so $\pk(\pi)\leq\des(\pi)$. For each peak $i$, note that $i-1\in[n-1]$ is not a descent, so that $\pk(\pi)\leq n-1-\des(\pi)$ and therefore $\des(\pi)\leq n-\pk(\pi)-1$. This proves (a).

To prove (b), it suffices to show that if
$n\geq1$, $0\leq j\leq\left\lfloor (n-1)/2\right\rfloor $,
and $j\leq k\leq n-j-1$ then there
exists a composition  of $n$ with $j$ non-final long parts (i.e.,
parts of size at least 2) and $k+1$ total parts. Such a composition
is  $(2^{j},1^{k-j},n-k-j)$. Hence, (b) is
proved.\end{proof}
\begin{prop}[Possible values of $(\lpk,\des)$]
\label{p-lpkdesvalues}
\leavevmode
\begin{enumerate}
\item [\normalfont{(a)}] For any permutation $\pi\in\mathfrak{P}_{n}$
with $n\geq1$, we have $0\leq\lpk(\pi)\leq\left\lfloor n/2\right\rfloor $.
In addition, if $\lpk(\pi)=0$, then $\des(\pi)=0$; otherwise, $\lpk(\pi)\leq\des(\pi)\leq n-\lpk(\pi)$.
\item [\normalfont{(b)}] If $n\geq1$, $1\leq j\leq\left\lfloor n/2\right\rfloor $,
and $j\leq k\leq n-j$, then there exists $\pi\in\mathfrak{P}_{n}$ with
$\lpk(\pi)=j$ and $\des(\pi)=k$. In addition, for any $n\geq1$,
there exists $\pi\in\mathfrak{P}_{n}$ with $\lpk(\pi)=\des(\pi)=0$.
\end{enumerate}
\end{prop}
\begin{proof}

If $\lpk(\pi)=0$,
then $\pi$ is an increasing permutation, so
we also have $\des(\pi)=0$. The other inequalities of part (a) follow from applying Proposition \ref{p-pkdesvalues} (a) to the permutation $0\pi$.

Now, fix $n\geq2$. (The case $n=1$ is obvious.) A permutation with descent composition $(n)$ has no left peaks and no descents. Suppose that $1\leq j\leq\left\lfloor n/2\right\rfloor $ and $j\leq k\leq n-j$. To complete the proof of (b), we show that there exists a composition $L$ of $n$ with exactly $k+1$ parts such that $\lpk(L)=\lr(L)+\sir(L)-\lfr(L)=j$. Such a composition is  $(1^{k-j+1},2^{j-1},n-k-j+1)$. This completes the proof of (b).\end{proof}

We say that $i\in[n-1]$ is an \textit{ascent} of an $n$-permutation $\pi$ if $\pi_i<\pi_{i+1}$. Let $\asc(\pi)$ denote the number of ascents of $\pi$. It is clear that $\des(\pi)=n-1-\asc(\pi)$.

\begin{prop}[Possible values of $(\udr,\des)$]
\label{p-udrdesvalues}
\leavevmode
\begin{enumerate}
\item [\normalfont{(a)}] For any permutation $\pi\in\mathfrak{P}_{n}$
with $n\geq1$, we have $1\leq\udr(\pi)\leq n$.
In addition, if $\udr(\pi)=1$, then $\des(\pi)=0$; otherwise, 
$\left\lfloor \udr(\pi)/2\right\rfloor \le \des(\pi) \le n - \left\lceil \udr(\pi)/2\right\rceil$.
\item [\normalfont{(b)}] If $n\geq1$, $2\leq j\leq n $,
and $\floor{j/2}\leq k\leq n-\ceil{j/2}$, then there exists $\pi\in\mathfrak{P}_{n}$ with
$\lpk(\pi)=j$ and $\des(\pi)=k$. In addition, for any $n\geq1$,
there exists $\pi\in\mathfrak{P}_{n}$ with $\udr(\pi)=1$ and $\des(\pi)=0$.
\end{enumerate}
\end{prop}

\begin{proof}
It is clear that every nonempty permutation has at least one up-down
run, and every up-down run of a permutation ends with a different
letter, so $1\leq\udr(\pi)\leq n$.
The beginning of the $2i$th up-down run of $\pi$ is always a descent of $\pi$, so $\des(\pi)\ge \floor{\udr(\pi)/2}$. The beginning of the $(2i-1)$th up-down run of $\pi$ is an ascent of $\pi$ for $i\ge2$, so the number of ascents of $\pi$ is at least $\floor{(\udr(\pi) -1)/2} = \ceil{\udr(\pi)/2}-1$. Thus 
\[\des(\pi) = n-1-\asc(\pi)\le n-1-\left(\ceil{\udr(\pi)/2}-1\right) 
= n-\ceil{\udr(\pi)/2},\]
completing the proof of (a).

Now, fix $n\geq2$. (The case $n=1$ is obvious.) A permutation with descent composition $(n)$
has only one up-down run and no descents. Suppose that $1\leq j\leq n$
and $\floor{j/2}\leq k\leq n-\ceil{j/2}$. To complete the proof of (b), we show that
there exists a composition $L$ of $n$ with exactly $k+1$ parts
such that 
$\udr(L)=\lpk(L)+\val(L)+1=2\sir(L)+2\lr(L)-\lfr(L)=j$.
For this, we consider three cases:
\begin{itemize}
\item If $j=2$, then we can take $(n-k,1^k)$.
\item If $j>2$ and $j$ is even, then we can take $(1,n-j/2-k+2, 2^{j/2-2}, 1^{k-j/2+1})$.
\item If $j$ is odd, then we can take $(1,1^{k-(j-1)/2},2^{(j-3)/2},n-(j+1)/2-k+2)$.
\end{itemize}
This completes the proof of (b).
\end{proof}

\subsection{A bijective proof of the shuffle-compatibility of the descent set}
\label{s-bijproof}

Here we give a simple proof that the descent set is a shuffle-compatible permutation statistic. The idea of the proof is inspired by the theory of $P$-partitions \cite{Stanley1972}.

Recall that in Section 2.1, we defined the inverse bijections $\Comp$
and $\Des$ between compositions of $n$ and subsets of $[n-1]$ in
the following way: for a set $A=\{a_{1},a_{2},\dots,a_{j}\}\subseteq[n-1]$
with $a_{1}<a_{2}<\cdots<a_{j}$, we let $\Comp(A)\coloneqq(a_{1},a_{2}-a_{1},\dots,a_{j}-a_{j-1},n-a_{j})$,
and for a composition $L=(L_{1},L_{2},\dots,L_{k})$, we let $\Des(L)\coloneqq\{L_{1},L_{1}+L_{2},\dots,L_{1}+\cdots+L_{k-1}\}$.
Observe that these maps extend to inverse bijections between weak compositions
of $n$ and multisubsets of $\{0\}\cup[n]$. (A weak composition allows 0 as a part.) For example, if $n=7$
and $A=\{0,2,2,5\}$, then $\Comp(A)=(0,2,0,3,2)$.

For two weak compositions $J=(J_{1},J_{2},\dots,J_{k})$ and $K=(K_{1},K_{2},\dots, K_{k})$
with the same number of parts, let $J+K$ denote the weak composition $(J_{1}+K_{1},J_{2}+K_{2},\dots,J_{k}+K_{k})$
obtained by summing the entries of $J$ and $K$ componentwise.
Also, we define the \textit{refinement
order} on weak compositions of $n$ analogously to
the refinement order on compositions of $n$; that is, $M$ covers
$L$ if and only if $M$ can be obtained from $L$ by replacing two
consecutive parts $L_{i}$ and $L_{i+1}$ with $L_{i}+L_{i+1}$. We
say that $L$ is a \textit{refinement} of $M$ if $L\leq M$ in the
refinement order.

\begin{lem}
\label{t-des-sc}
Let $\pi\in\mathfrak{P}_{m}$ and $\sigma\in\mathfrak{P}_{n}$ be
disjoint permutations, and let $A\subseteq[m+n-1]$ and $L=\Comp(A)$.
Then the number of shuffles of $\pi$ and $\sigma$ with descent set
contained in $A$ is equal to the number of weak compositions $J$
of $m$ and $K$ of $n$ such that $J$ is a refinement of $\Comp(\pi)$,
$K$ is a refinement of $\Comp(\sigma)$, $J$ and $K$ have the same
number of parts as $L$, and $J+K=L$.\end{lem}
\begin{proof}
Suppose that $L$ has $k$ parts, and let $J$ and $K$ satisfy the
above conditions. For every $i\in\Des(J)$, insert a bar immediately
before the $(i+1)$th letter of $\pi$.%
\footnote{Since $\Des(J)$ is a multiset, multiple bars may be inserted in any
given position.%
} Similarly, for every $i\in\Des(K)$, insert a bar immediately before
the $(i+1)$th letter of $\sigma$. This creates $k$ blocks of letters
in each of the permutations $\pi$ and $\sigma$ such that the letters
in each block are increasing. For example, take $\pi=12879$, $\sigma=4635$,
$A=\{1,5,6\}$, $L=(1,4,1,3)$, $J=(1,2,0,2)$, and $K=(0,2,1,1)$.
Then this yields the ``barred permutations'' $1|28||79$ and $|46|3|5$.

For each $1\leq i\leq k$, let $\tau^{(i)}$ denote the permutation
obtained by merging the letters in the $i$th block of $\pi$ and
the $i$th block of $\sigma$ in increasing order. Then let $\tau\in S(\pi,\sigma)$
be the concatenation $\tau^{(1)}\tau^{(2)}\cdots\tau^{(k)}$, which
has descent set contained in $A$. For example, using the $\pi$ and
$\sigma$ specified above, we have $\tau^{(1)}=1$, $\tau^{(2)}=2468$,
$\tau^{(3)}=3$, and $\tau^{(4)}=579$, so $\tau=124683579$. Since
$J+K=L$, the descent set of $\tau$ is contained in $A$.

To show that this procedure gives a bijection between shuffles of
$\pi$ and $\sigma$ with descent set contained in $A$ and pairs
of weak compositions $(J,K)$ satisfying the stated conditions, we
give an inverse procedure. Let $\tau\in S(\pi,\sigma)$ with $\Des(\tau)\subseteq A$,
and let $k=\left|A\right|+1$. For every $i\in A$, insert a bar after
the $i$th letter of $\tau$. Delete every letter in $\sigma$ from
$\tau$ to obtain the permutation $\pi$ decorated with bars, which
creates $k$ blocks of letters in $\pi$ such that the letters in
each block are increasing. Similarly, by deleting every letter in
$\pi$ from $\tau$, we obtain $k$ blocks of letters in $\sigma$
such that the letters in each block are increasing. Using the same
example as above, we begin with $A=\{1,5,6\}$ and $\tau=124683579$.
Inserting bars, we have $1|2468|3|579$, from which we obtain $1|28||79$
and $|46|3|5$. 

Now, for each $1\leq i\leq k$, let $J_{i}$ denote the size of the
$i$th block in $\pi$ and let $K_{i}$ denote the size of the $i$th
block in $\sigma$. Then define the weak compositions $J$ and $K$
by $J=(J_{1},J_{2},\dots,J_{k})$ and $K=(K_{1},K_{2},\dots,K_{k})$.
Continuing the example, we have $J=(1,2,0,2)$ and $K=(0,2,1,1)$.
Since the letters in every block are weakly increasing, $J$ is a
refinement of $\Comp(\pi)$ and $K$ is a refinement of $\Comp(\sigma)$.
Moreover, it is clear that $J$ and $K$ have the same number of parts
as $L=\Comp(A)$ and that $J+K=L$.
\end{proof}

Lemma \ref{t-des-sc} shows that the number of shuffles of $\pi$ and $\sigma$
with descent set contained in a specified set $A$ depends only on
$\Des(\pi)$, $\Des(\sigma)$, $\left|\pi\right|$, and $\left|\sigma\right|$.
By inclusion-exclusion, it follows that the number of shuffles of
$\pi$ and $\sigma$ with descent set equal to $A$ depends only on
$\Des(\pi)$, $\Des(\sigma)$, $\left|\pi\right|$, and $\left|\sigma\right|$.
In other words, the descent set is shuffle-compatible.

We can use the shuffle-compatibility of the descent set to prove the shuffle-compatibility of a family of related statistics that we call ``partial descent sets''. For non-negative integers $i$ and $j$, define the \textit{partial descent set} $\Des_{i,j}$ by
\[
\Des_{i,j}(\pi) \coloneqq \Des(\pi) \cap (\{1,2,\dots, i\}\cup\{n-1,\dots, n-j\}),
\]
where $n=\left|\pi\right|$. In other words,  
$\Des_{i,j}(\pi)$ is the set of descents of $\pi$ that occur in the first $i$ or last $j$ positions. For example, if $i+j\ge \left|\pi\right| -1$ then $\Des_{i,j}(\pi)=\Des(\pi)$, and for $\left|\pi\right|\ge2$,  $\left|\Des_{1,0}(\pi)\right| = \sir(\pi)$ and $\left|\Des_{0,1}(\pi)\right| = \sfr(\pi)$.

\begin{thm}
\label{t-truncDessc}
The partial descent sets $\Des_{i,j}$ for all $i,j\ge0$ are shuffle-compatible.
\end{thm}
\begin{proof}
We write $\Des_{i,j}(S(\pi,\sigma))$ for the multiset $\{\, \Des_{i,j}(\tau) \mid \tau\in S(\pi,\sigma)\,\}$. We define the equivalence relation $\equiv_{i,j}$ on permutations of the same length by 
$\pi \equiv_{i,j} \pi'$ if and only if $\Des_{i,j}(S(\pi,\sigma))=\Des_{i,j}(S(\pi',\sigma))$
for all $\sigma$ disjoint from both $\pi$ and $\pi'$.
(It is immediate from the above definition that $\equiv_{i,j}$ is reflexive and symmetric, and it is not hard to show that $\equiv_{i,j}$ is also transitive.) For $\pi$ and $\pi'$ in $\mathfrak{P}_m$, the following are  sufficient conditions for $\pi \equiv_{i,j} \pi'$:

\begin{enumerate}
\item[(i)] If $\pi$ and $\pi'$ have the same descent set, then $\pi \equiv_{i,j} \pi'$.
\item[(ii)] If $\pi_k=\pi_k'$ for all $k$ with  $1 \leq k \leq i+1$ or $m-j \leq k \leq m$, then $\pi \equiv_{i,j} \pi'$.
\end{enumerate}

Condition (i) is a consequence of the shuffle-compatibility of the descent set. Condition (ii) follows from the fact that $\Des_{i,j}(\tau)$ for $\tau\in S(\pi,\sigma)$ does not depend on  the values of $\pi_k$ with $i+1<k<m-j$.%
\footnote{This is because if $i+1<k<m-j$, then upon shuffling $\pi$ with any permutation $\sigma$ disjoint from $\pi$, the letter $\pi_k$ cannot end up in the first $i+1$ or last $j+1$ positions of any element of $S(\pi,\sigma)$.}

We claim that to prove the theorem it is sufficient to show that $\Des_{i,j}(\pi) = \Des_{i,j}(\pi')$ implies $\pi \equiv_{i,j} \pi'$. Indeed, let $\pi$ and $\pi'$ be two permutations of the same length with $\Des_{i,j}(\pi)=\Des_{i,j}(\pi')$ and similarly with $\sigma$ and $\sigma'$, where $\pi$ is disjoint from $\sigma$ and  $\pi'$ is disjoint from $\sigma'$. By (i), we can assume without loss of generality that $\sigma$ is disjoint from $\pi'$ as well, and thus if we have $\pi \equiv_{i,j} \pi'$ and $\sigma \equiv_{i,j} \sigma'$, then  $\Des_{i,j}(S(\pi,\sigma))=\Des_{i,j}(S(\pi',\sigma))=\Des_{i,j}(S(\pi',\sigma'))$.

Now suppose that $\pi$ and $\pi'$ are in $\mathfrak{P}_m$ with $\Des_{i,j}(\pi) =\Des_{i,j}(\pi')$. We shall show that $\pi \equiv_{i,j} \pi'$, considering three cases separately:

\begin{enumerate}
\item First, suppose that $i+j\ge m-1$. Then $\Des(\pi) = \Des(\pi')$, so $\pi \equiv_{i,j} \pi'$ by (i). 

\item Next, suppose that $i+j\le m-3$. It is enough to find permutations $\bar\pi$ and $\bar\pi'$ such that $\bar\pi \equiv_{i,j} \pi$, $\bar\pi' \equiv_{i,j} \pi'$, and $\Des(\bar\pi)=\Des(\bar\pi')$. 
To do this, we may choose some $a\in\mathbb{P}$ greater than all the letters of $\pi$ and $\pi'$ and construct $\bar\pi$ and $\bar\pi'$ by replacing the letters in positions $i+2, i+3,\dots, m-j-1$ of both $\pi$ and $\pi'$ with the sequence $a \mskip 6mu a+1\,\cdots\, a+(m-i-j-3)$.

\item Finally, suppose that $i+j=m-2$. In this case, $\Des_{i,j}(\pi)$ comprises all descents of $\pi $ except in position $i+1$, so that 
$\Des(\pi)$ and $\Des(\pi')$ are the same except that $i+1$ may be in one but not the other. If $\Des(\pi)=\Des(\pi')$ then $\pi \equiv_{i,j} \pi'$, so let us suppose that $i+1$ is a descent of exactly one of $\pi$ and $\pi'$.
Let $\sigma\in\mathfrak{P}_n$ be disjoint from $\pi$. By (i), we may assume without loss of generality that no letter of $\pi$ or $\sigma$ has value strictly between $\pi_{i+1}$ and $\pi_{i+2}$. Let $\pi^*$ be the result of switching  $\pi_{i+1}$ and $\pi_{i+2}$ in $\pi$.  
It is easy to see that switching $\pi_{i+1}$ and $\pi_{i+2}$ in an element $\tau$ of $S(\pi,\sigma)$ can change $\Des(\tau)$ only by adding or removing a single descent which is at least $i+1$ and at most  $n+i+1=m+n-1-j$ and thus does not change $\Des_{i,j}(\tau)$. Thus, $\pi^* \equiv_{i,j} \pi$. Since $\Des(\pi^*) = \Des(\pi')$, we also have $\pi^* \equiv_{i,j} \pi'$, so $\pi \equiv_{i,j} \pi'$ as desired.\qedhere
\end{enumerate}
\end{proof}

\section{Shuffle algebras}
\label{s-section3}

\subsection{Definition and basic results}

Every permutation statistic $\st$ induces an equivalence relation
on permutations; we say that permutations $\pi$ and $\sigma$ are
$\st$-\textit{equivalent} if $\st(\pi)=\st(\sigma)$ and $\left|\pi\right|=\left|\sigma\right|$.\footnote{The notion of $\st$-equivalence should not be confused with that of ``$\st$-Wilf equivalence'' \cite{Dokos2012}. }
We write the $\st$-equivalence class of $\pi$ as $[\pi]_{\st}$.
For a shuffle-compatible statistic $\st$, we can then associate to
$\st$ a $\mathbb{Q}$-algebra in the following way. First, associate
to $\st$ a $\mathbb{Q}$-vector space by taking as a basis the $\st$-equivalence
classes of permutations. We give this vector space a multiplication
by taking 
\[
[\pi]_{\st}[\sigma]_{\st}=\sum_{\tau\in S(\pi,\sigma)}[\tau]_{\st},
\]
which is well-defined (i.e., the choice of $\pi$ and $\sigma$ in an equivalence class does
not matter) because $\st$ is shuffle-compatible. Conversely, if such
a multiplication is well-defined, then $\st$ is shuffle-compatible.
We denote the resulting algebra by ${\cal A}_{\st}$ and call it the
\textit{shuffle algebra} of $\st$. Observe that ${\cal A}_{\st}$
is graded, and $[\pi]_{\st}$ belongs to the $n$th homogeneous component
of ${\cal A}_{\st}$ if $\pi$ has length $n$.

As an example, we describe the shuffle algebra of the major index
$\maj$.
\begin{thm}[Shuffle-compatibility of the major index]
\label{t-majsc} \leavevmode
\begin{itemize}
\item [\normalfont{(a)}] The major index $\maj$ is shuffle-compatible.
\item [\normalfont{(b)}] The linear map on ${\cal A}_{\maj}$ defined by
\[
[\pi]_{\maj}\mapsto\frac{q^{\maj(\pi)}}{[\left|\pi\right|]_{q}!}x^{\left|\pi\right|}
\]
is a $\mathbb{Q}$-algebra isomorphism from ${\cal A}_{\maj}$ to
the span of 
\[
\left\{ \frac{q^{j}}{[n]_{q}!}x^{n}\right\} _{n\geq0,\:0\leq j\leq{n \choose 2}},
\]
a subalgebra of $\mathbb{Q}[[q]][x]$.
\item [\normalfont{(c)}] The $n$th homogeneous component of ${\cal A}_{\maj}$
has dimension ${n \choose 2}+1$.%
\end{itemize}
\end{thm}
\begin{proof}
We know from (\ref{e-maj}) that $\maj$ is shuffle-compatible, so
there is no need to prove (a). Let $\phi \colon {\cal A}_{\maj}\rightarrow\mathbb{Q}[[q]][x]$
denote the map given in the statement of (b). Then by (\ref{e-maj}),
for $\pi\in\mathfrak{P}_{m}$ and $\sigma\in\mathfrak{P}_{n}$, we
have 
\begin{align*}
\phi([\pi]_{\maj})\phi([\sigma]_{\maj}) & =\frac{q^{\maj(\pi)}}{[m]_{q}!}x^{m}\frac{q^{\maj(\sigma)}}{[n]_{q}!}x^{n}\\
 & =\frac{q^{\maj(\pi)+\maj(\sigma)}}{[m]_{q}![n]_{q}!}x^{m+n}\\
 & =\frac{q^{\maj(\pi)+\maj(\sigma)}}{[m+n]_{q}!}
\qbinom{m+n}{m}x^{m+n}\\
 & =\sum_{\tau\in S(\pi,\sigma)}\frac{q^{\maj(\tau)}}{[m+n]_{q}!}x^{m+n}\\
 & =\phi([\pi]_{\maj}[\sigma]_{\maj}),
\end{align*}
so $\phi$ is an algebra homomorphism. The possible values for $\maj(\pi)$
for an $n$-permutation $\pi$ range from 0 to ${n \choose 2}$, and
since the elements $q^{j}x^{n}/[n]_{q}!$ are linearly independent,
$\phi$ gives an isomorphism from ${\cal A}_{\maj}$ to the stated
subalgebra, thus proving (b) and (c).
\end{proof}
We say that two permutation statistics $\st_{1}$ and $\st_{2}$ are
\textit{equivalent} if $[\pi]_{\st_{1}}=[\pi]_{\st_{2}}$ for every
permutation $\pi$. In other words, $\st_{2}(\pi)$ depends only on
$\st_{1}(\pi)$ and $\left|\pi\right|$ for every permutation $\pi$,
and vice versa. As shown in Lemma \ref{l-udr}, $\udr$ and $(\lpk,\val)$
are equivalent statistics. It also follows from the formula $\comaj(\pi)=n\des(\pi)-\maj(\pi)$
that $(\des,\maj)$ and $(\des,\comaj)$ are equivalent statistics.
\begin{thm}
\label{t-esc} Suppose that $\st_{1}$ and $\st_{2}$ are equivalent
statistics. If $\st_{1}$ is shuffle-compatible with shuffle algebra
${\cal A}_{\st_{1}}$, then $\st_{2}$ is also shuffle-compatible
with shuffle algebra ${\cal A}_{\st_{2}}$ isomorphic to ${\cal A}_{\st_{1}}$.\end{thm}
\begin{proof}
Equivalent statistics have the same equivalence classes on permutations,
so ${\cal A}_{\st_{1}}$ and ${\cal A}_{\st_{2}}$ (as vector spaces)
have the same basis elements. If $\st_{1}$ and $\st_{2}$ are equivalent,
then 
\[
[\pi]_{\st_{2}}[\sigma]_{\st_{2}}=[\pi]_{\st_{1}}[\sigma]_{\st_{1}}=\sum_{\tau\in S(\pi,\sigma)}[\tau]_{\st_{1}}=\sum_{\tau\in S(\pi,\sigma)}[\tau]_{\st_{2}},
\]
which proves the result.
\end{proof}

For example, it is easy to see that $\Des_{1,0}$ is equivalent to $\sir$ and that $\Des_{0,1}$ is equivalent to $\sfr$. Thus, Theorem \ref{t-truncDessc} implies that $\sir$ and $\sfr$ are shuffle-compatible as well.

We say that $\st_1$ is a \textit{refinement} of $\st_2$ if for all permutations $\pi$ and $\sigma$ of the same length, $\st_1(\pi) = \st_1(\sigma)$ implies $\st_2(\pi) = \st_2(\sigma)$. For example, the statistics of which the descent set is a refinement  are exactly what we call descent statistics.

\begin{thm}
\label{t-quots}
Suppose that $\st_1$ is shuffle-compatible and is a refinement of $\st_2$. Let $A$ be  a $\mathbb{Q}$-algebra with basis $\{u_{\alpha}\}$ indexed by $\st_2$-equivalence classes $\alpha$, and suppose that there exists a $\mathbb{Q}$-algebra homomorphism $\phi\colon \mathcal{A}_{\st_1}\to A$ such that for every $\st_1$-equivalence class $\beta$, we have $\phi(\beta) = u_\alpha$ where $\alpha$ is the $\st_2$-equivalence class containing $\beta$. Then $\st_2$ is shuffle-compatible and the map $u_\alpha\mapsto\alpha$ extends by linearity to an isomorphism from $A$ to ${\cal A}_{\st_2}$.
\end{thm}
\begin{proof}
It is sufficient to show that for any two disjoint permutations $\pi$ and $\sigma$, we have
\begin{equation*}
u_{[\pi]_{\st_2}}u_{[\sigma]_{\st_2}}
  =\sum_{\tau\in S(\pi, \sigma)} u_{[\tau]_{\st_2}}.
\end{equation*}
To see this, we have
\begin{align*}
u_{[\pi]_{\st_2}}u_{[\sigma]_{\st_2}} 
  &= \phi([\pi]_{\st_1})\phi([\sigma]_{\st_1})\\
  &= \phi([\pi]_{\st_1}[\sigma]_{\st_1})\\
  &=\phi\Big(\sum_{\tau\in S(\pi, \sigma)} [\tau]_{\st_1}\Big)\\
  &=\sum_{\tau\in S(\pi, \sigma)} u_{[\tau]_{\st_2}}.\qedhere
\end{align*}
\end{proof}

\subsection{Basic symmetries yield isomorphic shuffle algebras}

Here we consider three involutions on permutations given by symmetries\textemdash reversion,
complementation, and reverse-complementation\textemdash and their
implications for the shuffle-compatibility of permutation statistics.

Given $\pi=\pi_{1}\pi_{2}\cdots\pi_{n}\in\mathfrak{P}_{n}$, we define
the \textit{reversal} $\pi^{r}$ of $\pi$ to be $\pi^{r}\coloneqq\pi_{n}\pi_{n-1}\cdots\pi_{1}$, 
the \textit{complement} $\pi^{c}$ of $\pi$ to be the permutation
obtained by (simultaneously) replacing the $i$th smallest letter
in $\pi$ with the $i$th largest letter in $\pi$ for all $1\leq i\leq n$,
and the \textit{reverse-complement} $\pi^{rc}$ of $\pi$ to be $\pi^{rc}\coloneqq(\pi^{r})^{c}=(\pi^{c})^{r}$.
For example, given $\pi=139264$, we have $\pi^{r}=462931$, $\pi^{c}=941623$,
and $\pi^{rc}=326149$.

More generally, let $f$ be an involution on the set of permutations
which preserves the length of a permutation. Then let $\pi^{f}$ denote
$f(\pi)$. Given a set $X$ of permutations, let 
\[
X^{f}\coloneqq\{\,\pi^{f}\mid\pi\in X\,\},
\]
so $f$ naturally induces an involution on sets of permutations as well.

We say that two permutation statistics $\st_{1}$ and $\st_{2}$ are
\textit{$f$-equivalent} if $\st_{1}\circ f$ is equivalent to $\st_{2}$.
Equivalently, $\st_{1}$ and $\st_{2}$ are $f$-equivalent if $([\pi^{f}]_{\st_{1}})^{f}=[\pi]_{\st_{2}}$ for all $\pi$.
It is easy to verify that $\st_{1}(\pi^{f})=\st_{2}(\pi)$ implies
that $\st_{1}$ and $\st_{2}$ are $f$-equivalent (although this
is not a necessary condition).

For example, $\Lpk$ and $\Rpk$ are $r$-equivalent, $\pk$ and $\val$
are $c$-equivalent, $\Pk$ and $\Val$ are $c$-equivalent, $(\pk,\des)$
and $(\val,\des)$ are $rc$-equivalent, and $\maj$ and $\comaj$
are $rc$-equivalent. It is less obvious that $(\lpk,\val)$ and $(\lpk,\pk)$
are $rc$-equivalent, so we provide a proof below.
\begin{prop}
\label{p-lpkpk} $(\lpk,\val)$ and $(\lpk,\pk)$ are $rc$-equivalent
statistics.
\end{prop}
\begin{proof}
Fix a permutation $\pi$. We divide into four cases: (a) $\pi$ has
a short initial run and a long final run, (b) $\pi$ has a short initial
run and a short final run, (c) $\pi$ has a long initial run and a
long final run, and (d) $\pi$ has a long initial run and short final
run. In case (a), we know from Lemma \ref{l-udr} that $\lpk(\pi)=\val(\pi)$.
Then $\pk(\pi^{rc})=\val(\pi)$, and $\pi^{rc}$ has a long initial
run, so 
\[
\lpk(\pi^{rc})=\pk(\pi^{rc})=\val(\pi)=\lpk(\pi).
\]
Thus, $(\lpk,\val)(\pi)=(\lpk,\pk)(\pi^{rc})$. The other three cases
can be verified in the same way.
\end{proof}

Let us say that $f$ is \textit{shuffle-compatibility-preserving} if
for every pair of disjoint permutations $\pi$ and $\sigma$, there
exist disjoint permutations $\hat{\pi}$ and $\hat{\sigma}$ with
the same relative order as $\pi$ and $\sigma$, respectively, such
that $S(\hat{\pi}^{f},\hat{\sigma}^{f})=S(\pi,\sigma)^{f}$ and $S(\pi^{f},\sigma^{f})=S(\hat{\pi},\hat{\sigma})^{f}$.

We note that $f$-equivalences are not actually equivalence relations on statistics (although they are symmetric), but we shall show that if the statistics are shuffle-compatible and $f$ is shuffle-compatibility-preserving, then $f$-equivalences induce isomorphisms on the corresponding shuffle algebras.

\begin{thm}
\label{t-fsc} Let $f$ be shuffle-compatibility-preserving, and suppose
that $\st_{1}$ and $\st_{2}$ are $f$-equivalent statistics. If
$\st_{1}$ is shuffle-compatible with shuffle algebra ${\cal A}_{\st_{1}}$,
then $\st_{2}$ is also shuffle-compatible with shuffle algebra ${\cal A}_{\st_{2}}$
isomorphic to ${\cal A}_{\st_{1}}$.
\end{thm}
\begin{proof}
Let $\pi$ and $\bar{\pi}$ be permutations in the same $\st_{2}$-equivalence
class and similarly with $\sigma$ and $\bar{\sigma}$, such that
$\pi$ and $\sigma$ are disjoint and $\bar{\pi}$ and $\bar{\sigma}$
are disjoint. Since $\st_{1}$ and $\st_{2}$ are $f$-equivalent,
it follows that 
\[
([\pi^{f}]_{\st_{1}})^{f}=[\pi]_{\st_{2}}=[\bar{\pi}]_{\st_{2}}=([\bar{\pi}^{f}]_{\st_{1}})^{f}.
\]
Hence $[\pi^{f}]_{\st_{1}}=[\bar{\pi}^{f}]_{\st_{1}}$ and similarly
$[\sigma^{f}]_{\st_{1}}=[\bar{\sigma}^{f}]_{\st_{1}}$. 

Since $f$ is shuffle-compatibility-preserving, there exist permutations
$\hat{\pi},$ $\hat{\sigma}$, $\hat{\bar{\pi}}$, and $\hat{\bar{\sigma}}$\textemdash having
the same relative order as $\pi$, $\sigma$, $\bar{\pi}$, and $\bar{\sigma}$,
respectively\textemdash satisfying $S(\hat{\pi}^{f},\hat{\sigma}^{f})=S(\pi,\sigma)^{f}$,
$S(\pi^{f},\sigma^{f})=S(\hat{\pi},\hat{\sigma})^{f}$, $S(\hat{\bar{\pi}}^{f},\hat{\bar{\sigma}}^{f})=S(\bar{\pi},\bar{\sigma})^{f}$,
and $S(\bar{\pi}^{f},\bar{\sigma}^{f})=S(\hat{\bar{\pi}},\hat{\bar{\sigma}})^{f}$.
By the ``same relative order'' property, we have
\[
[\hat{\pi}^{f}]_{\st_{1}}=[\pi^{f}]_{\st_{1}}=[\bar{\pi}^{f}]_{\st_{1}}=[\hat{\bar{\pi}}^{f}]_{\st_{1}}
\]
and 
\[
[\hat{\sigma}^{f}]_{\st_{1}}=[\sigma^{f}]_{\st_{1}}=[\bar{\sigma}^{f}]_{\st_{1}}=[\hat{\bar{\sigma}}^{f}]_{\st_{1}}.
\]

Now, by shuffle-compatibility of $\st_{1}$, we have the equality
of multisets 
\[
\{\,\st_{1}(\tau)\mid\tau\in S(\hat{\pi}^{f},\hat{\sigma}^{f})\,\}=\{\,\st_{1}(\tau)\mid\tau\in S(\hat{\bar{\pi}}^{f},\hat{\bar{\sigma}}^{f})\,\},
\]
which is equivalent to 
\[
\{\,\st_{2}(\tau)\mid\tau^{f}\in S(\hat{\pi}^{f},\hat{\sigma}^{f})\,\}=\{\,\st_{2}(\tau)\mid\tau^{f}\in S(\hat{\bar{\pi}}^{f},\hat{\bar{\sigma}}^{f})\,\}
\]
by $f$-equivalence of $\st_{1}$ and $\st_{2}$, and from $S(\hat{\pi}^{f},\hat{\sigma}^{f})=S(\pi,\sigma)^{f}$
and $S(\hat{\bar{\pi}}^{f},\hat{\bar{\sigma}}^{f})=S(\bar{\pi},\bar{\sigma})^{f}$,
we have
\[
\{\,\st_{2}(\tau)\mid\tau\in S(\pi,\sigma)\,\}=\{\,\st_{2}(\tau)\mid\tau\in S(\bar{\pi},\bar{\sigma})\,\}.
\]
Therefore, $\st_{2}$ is shuffle-compatible.

It remains to prove that ${\cal A}_{\st_{2}}$ is isomorphic to ${\cal A}_{\st_{1}}$.
Observe that 
\[
\sum_{\tau\in S(\hat{\pi},\hat{\sigma})}[\tau]_{\st_{2}}=\sum_{\tau\in S(\pi,\sigma)}[\tau]_{\st_{2}},
\]
since $\st_{2}$ is shuffle-compatible. Define the linear
map $\varphi_{f} \colon {\cal A}_{\st_{2}}\rightarrow{\cal A}_{\st_{1}}$
by $[\pi]_{\st_{2}}\mapsto[\pi^{f}]_{\st_{1}}$. Then
\begin{align*}
\varphi_{f}([\pi]_{\st_{2}}[\sigma]_{\st_{2}}) & =\varphi_{f}\Big(\sum_{\tau\in S(\pi,\sigma)}[\tau]_{\st_{2}}\Big)\\
 & =\sum_{\tau\in S(\pi,\sigma)}\varphi_{f}([\tau]_{\st_{2}})\\
 & =\sum_{\tau\in S(\pi,\sigma)}[\tau^{f}]_{\st_{1}}\\
 & =\sum_{\tau\in S(\hat{\pi},\hat{\sigma})}[\tau^{f}]_{\st_{1}}\\
 & =\sum_{\tau\in S(\hat{\pi},\hat{\sigma})^{f}}[\tau]_{\st_{1}}\\
 & =\sum_{\tau\in S(\pi^{f},\sigma^{f})}[\tau]_{\st_{1}}\\
 & =[\pi^{f}]_{\st_{1}}[\sigma^{f}]_{\st_{1}}\\
 & =\varphi_{f}([\pi]_{\st_{2}})\varphi_{f}([\sigma]_{\st_{2}}),
\end{align*}
so $\varphi_{f}$ is an isomorphism from ${\cal A}_{\st_{2}}$ to ${\cal A}_{\st_{1}}$.
\end{proof}
\begin{lem}
Reversion, complementation, and reverse-complementation are shuffle-com\-pat\-i\-bil\-ity-preserving.
\end{lem}
\begin{proof}
It is clear that $S(\pi^{r},\sigma^{r})=S(\pi,\sigma)^{r}$, so by
taking $\hat{\pi}=\pi$ and $\hat{\sigma}=\sigma$, the equalities
$S(\hat{\pi}^{r},\hat{\sigma}^{r})=S(\pi,\sigma)^{r}$ and $S(\pi^{r},\sigma^{r})=S(\hat{\pi},\hat{\sigma})^{r}$
come for free. Thus reversion is shuffle-compatibility-preserving.

Unlike with reversion, it is not true in general that $S(\pi^{c},\sigma^{c})=S(\pi,\sigma)^{c}$.
For disjoint permutations $\pi=\pi_{1}\pi_{2}\cdots\pi_{m}$ and $\sigma=\sigma_{1}\sigma_{2}\cdots\sigma_{n}$,
let $P=\{\pi_{1},\dots,\pi_{m},\sigma_{1},\dots,\sigma_{n}\}$ be
the set of letters appearing in $\pi$ and $\sigma$, and let $\rho:P\rightarrow P$
be the map sending the $i$th smallest letter of $P$ to the $i$th
largest letter of $P$ for every $i$. By an abuse of notation, let
$\rho(\pi)$ denote the permutation $\rho(\pi_{1})\rho(\pi_{2})\cdots\rho(\pi_{m})$
obtained by applying $\rho$ to each letter in $\pi$. Then, let $\hat{\pi}=\rho(\pi^{c})$
and $\hat{\sigma}=\rho(\sigma^{c})$. For example, let $\pi=413$
and $\sigma=25$. Then $P=[5]$, $\pi^{c}=143$, and $\sigma^{c}=52$,
and so $\hat{\pi}=523$ and $\hat{\sigma}=14$. Clearly, $\pi$ has
the same relative order as $\hat{\pi}$, and similarly with $\sigma$
and $\hat{\sigma}$. It is also easy to see that $\rho(\pi)=\widehat{\pi^{c}}=\hat{\pi}^{c}$
and $\rho(\sigma)=\widehat{\sigma^{c}}=\hat{\sigma}^{c}$.

To see that $S(\hat{\pi}^{c},\hat{\sigma}^{c})=S(\pi,\sigma)^{c}$,
first let $\tau\in S(\pi,\sigma)$. Then $\tau$ contains both $\pi$
and $\sigma$ as subsequences, and to show that $\tau^{c}\in S(\hat{\pi}^{c},\hat{\sigma}^{c})$,
it suffices to show that $\tau^{c}$ contains both $\hat{\pi}^{c}=\rho(\pi)$
and $\hat{\sigma}^{c}=\rho(\sigma)$ as subsequences. However, this
follows from the fact that, when taking the complement of $\tau$,
the subsequence $\pi$ appearing in $\tau$ is transformed into $\rho(\pi)$,
and similarly $\sigma$ turns into $\rho(\sigma)$. The other inclusion
follows by the same reasoning, and the equality $S(\pi^{c},\sigma^{c})=S(\hat{\pi},\hat{\sigma})^{c}$
follows directly from $S(\hat{\pi}^{c},\hat{\sigma}^{c})=S(\pi,\sigma)^{c}$
and replacing $\pi$ and $\sigma$ with $\pi^{c}$ and $\sigma^{c}$,
respectively. Hence complementation is shuffle-compatibility-preserving.

Finally, the equalities $S(\pi^{r},\sigma^{r})=S(\pi,\sigma)^{r}$,
$S(\hat{\pi}^{c},\hat{\sigma}^{c})=S(\pi,\sigma)^{c}$, and $S(\pi^{c},\sigma^{c})=S(\hat{\pi},\hat{\sigma})^{c}$
imply $S(\hat{\pi}^{rc},\hat{\sigma}^{rc})=S(\pi,\sigma)^{rc}$ and
$S(\pi^{rc},\sigma^{rc})=S(\hat{\pi},\hat{\sigma})^{rc}$. Thus reverse-complementation
is shuffle-compatibility-preserving.
\end{proof}
\begin{cor} \label{c-rcsc}
Suppose that $\st_{1}$ and $\st_{2}$ are $r$-equivalent, $c$-equivalent,
or $rc$-equivalent statistics. If $\st_{1}$ is shuffle-compatible
with shuffle algebra ${\cal A}_{\st_{1}}$, then $\st_{2}$ is also
shuffle-compatible with shuffle algebra ${\cal A}_{\st_{2}}$ isomorphic
to ${\cal A}_{\st_{1}}$.
\end{cor}
For example, since $\maj$ and $\comaj$ are $rc$-equivalent, it
follows from Theorem \ref{t-majsc} and Corollary \ref{c-rcsc} that $\comaj$
is shuffle-compatible and that its shuffle algebra ${\cal A}_{\comaj}$
is isomorphic to ${\cal A}_{\maj}$.

\subsection{A note on Hadamard products}

The operation
of \textit{Hadamard product} $*$ on formal power series in $t$ is given by 
\[
\biggl(\sum_{n=0}^{\infty}a_{n}t^{n}\biggr)*\biggl(\sum_{n=0}^{\infty}b_{n}t^{n}\biggr)\coloneqq\sum_{n=0}^{\infty}a_{n}b_{n}t^{n}.
\]

Many shuffle algebras that we study in this paper can be characterized as subalgebras of various algebras in which the multiplication is the Hadamard product in a variable $t$. In the notation for these algebras, we write $t*$ to indicate that multiplication is the Hadamard product in $t$. For example, $\mathbb{Q}[[t*,q]][x]$ is the algebra of polynomials in $x$ whose coefficients are formal power series in $t$ and $q$, where  multiplication is  ordinary multiplication in the variables $x$ and $q$ but is the Hadamard product in $t$.

We note that the Hadamard product is only used in descriptions of shuffle algebras and in the proof of Lemma \ref{l-monoidlike2}, where $t^m * t^n$ denotes the Hadamard product of $t^m$ and $t^n$. (Here, $t^m$ is the ordinary product of $m$ copies of $t$ and similarly with $t^n$.) All other expressions should be interpreted as using ordinary multiplication. For instance, any expression with an exponent such as $t^k$ or $(1+yt)^k$ is ordinary multiplication, and $(1-tf)^{-1}$ (as in Corollary \ref{c-monoidlike2}) denotes
$\sum_{k=0}^\infty t^k f^k$.

\section{Quasisymmetric functions and shuffle-compatibility}
\label{s-section4}

\subsection{The descent set shuffle algebra \texorpdfstring{$\QSym$}{QSym}}

A formal power series $f\in\mathbb{Q}[[x_{1},x_{2},\dots]]$ of bounded degree in countably
many commuting variables $x_{1},x_{2},\dots$  is
called a \textit{quasisymmetric function} if for any
positive integers $a_{1},a_{2},\dots,a_{k}$, if 
$i_{1}<i_{2}<\cdots<i_{k}$ and $j_{1}<j_{2}<\cdots<j_{k}$, then
\[
[x_{i_{1}}^{a_{1}}x_{i_{2}}^{a_{2}}\cdots x_{i_{k}}^{a_{k}}]\,f=[x_{j_{1}}^{a_{1}}x_{j_{2}}^{a_{2}}\cdots x_{j_{k}}^{a_{k}}]\,f.
\]

It is clear that every symmetric function is quasisymmetric, but not
every quasisymmetric function is symmetric. For example, $\sum_{i<j<k}x_{i}^{2}x_{j}x_{k}$
is quasisymmetric, but it is not symmetric because $x_{1}^{2}x_{2}x_{3}$
appears as a term yet $x_{1}x_{2}^{2}x_{3}$ does not.

Let $L \vDash n$ indicate that $L$ is a composition of $n$, and let $\QSym_{n}$ be the set of quasisymmetric functions homogeneous
of degree $n$, which is clearly a vector space. For a composition $L=(L_1,L_2,\dots, L_k)$, the \emph{monomial quasisymmetric function} $M_L$ is defined by
\begin{equation*}
M_L \coloneqq \sum_{i_1<i_2<\dots<i_k}x_{i_1}^{L_1}x_{i_2}^{L_2}\dots x_{i_k}^{L_k}
\end{equation*}
It is clear that $\{M_L\}_{L\vDash n}$ is a basis for $\QSym_n$, so  for $n\ge1$, 
$\QSym_n$ has dimension $2^{n-1}$, the number of compositions of $n$. 

Another important basis for $\QSym_{n}$ (and the
most important basis for our purposes) is the basis of \textit{fundamental
quasisymmetric functions} $\{F_{L}\}_{L\vDash n}$ given by 
\[
F_{L}\coloneqq\sum_{\substack{i_{1}\leq i_{2}\leq\cdots\leq i_{n}\\
i_{j}<i_{j+1}\,\mathrm{if}\, j\in\Des(L)
}
}x_{i_{1}}x_{i_{2}}\cdots x_{i_{n}}.
\]
It is easy to see that 
\begin{equation}
\label{e-FtoM}
F_L = \sum_{\substack{\Des(K)\supseteq \Des(L)\\|K|=|L|}} M_K,
\end{equation}
so by inclusion-exclusion, $M_K$ can be expressed as a linear combination of the $F_L$. It follows that $\{F_{L}\}_{L\vDash n}$ spans $\QSym_n$, so this set must be a basis for $\QSym_n$ since it has the correct number of elements.

The product of two quasisymmetric functions is quasisymmetric, with
the product formula for the fundamental basis given by the following
theorem, which may be proved using $P$-partitions; see \cite[Exercise 7.93]{Stanley2001}.
This theorem may also be derived from Lemma \ref{t-des-sc}.

\begin{thm}
\label{t-fqsym} Let $c_{J,K}^{L}$ be the number of permutations
with descent composition $L$ among the shuffles of a permutation $\pi$ 
with descent composition $J$ and a permutation $\sigma$ \emph{(}disjoint from $\pi$\emph{)} with descent composition $K$. Then 
\begin{equation}
F_{J}F_{K}=\sum_{L}c_{J,K}^{L}F_{L}.\label{e-fundshuffle}
\end{equation}
\end{thm}

If $f\in\QSym_{m}$ and $g\in\QSym_{n}$, then $fg\in\QSym_{m+n}$.
Thus $\QSym\coloneqq\bigoplus_{n=0}^{\infty}\QSym_{n}$ is a graded
$\mathbb{Q}$-algebra called the \textit{algebra of quasisymmetric
functions} with coefficients in $\mathbb{Q}$, a subalgebra of $\mathbb{Q}[[x_{1},x_{2},\dots]]$. Motivated by Richard Stanley's theory of $P$-partitions, the first
author introduced quasisymmetric functions in \cite{Gessel1984} and
developed the basic algebraic properties of $\QSym$. Further properties of $\QSym$ and connections with many topics of study in combinatorics and algebra were developed in the subsequent decades.
Basic references include \cite[Section 7.19]{Stanley2001}, \cite[Section 5]{Grinberg2014},
and \cite{Luoto2013}.

Observe that Theorem \ref{t-fqsym} implies that $\QSym$ is isomorphic
to the shuffle algebra for the descent set with the fundamental basis
corresponding to the basis of $\Des$-equivalence classes.

\goodbreak
\begin{cor}[Shuffle-compatibility of the descent set]
\leavevmode
\begin{itemize}
\item [\normalfont{(a)}] The descent set $\Des$ is shuffle-compatible. 
\item [\normalfont{(b)}] The linear map on ${\cal A}_{\Des}$ defined by 
\[
[\pi]_{\Des}\mapsto F_{\Comp(\pi)}
\]
is a $\mathbb{Q}$-algebra isomorphism from ${\cal A}_{\Des}$ to $\QSym$.
\end{itemize}

\end{cor}
Now, let $\st$ be a descent statistic. Then not only does $\st$
induce a equivalence relation on permutations, but it also induces
a equivalence relation on compositions because permutations with the
same descent composition are necessarily $\st$-equivalent.

We establish
a necessary and sufficient condition for the shuffle-compatibility
of a descent statistic, which will also imply that the shuffle algebra
of any shuffle-compatible descent statistic is isomorphic to a quotient
of $\QSym$.
\begin{thm}
\label{t-scqsym} 
A descent statistic $\st$ is shuffle-compatible
if and only if there exists a $\mathbb{Q}$-algebra homomorphism $\phi_{\st}\colon\QSym\rightarrow A$,
where $A$ is a $\mathbb{Q}$-algebra with basis $\{u_{\alpha}\}$
indexed by $\st$-equivalence classes $\alpha$ of compositions, such
that $\phi_{\st}(F_{L})=u_{\alpha}$ whenever $L\in\alpha$. In this
case, the linear map on ${\cal A}_{\st}$ defined by 
\[
[\pi]_{\st}\mapsto u_{\alpha},
\]
where $\Comp(\pi)\in\alpha$, is a $\mathbb{Q}$-algebra isomorphism
from ${\cal A}_{\st}$ to $A$.
\end{thm}
\begin{proof}
Suppose that $\st$ is a shuffle-compatible descent statistic. Let $A=\mathcal{A}_{\st}$ be the shuffle algebra of $\st$, and let $u_{\alpha} = [\pi]_{\st}$ for any $\pi$ satisfying $\Comp(\pi)\in\alpha$, so that
\[
u_{\beta}u_{\gamma}=\sum_{\alpha}c_{\beta,\gamma}^{\alpha}u_{\alpha}
\]
where $c_{\beta,\gamma}^{\alpha}$ is the number of permutations with
descent composition in $\alpha$ that are obtained as a shuffle of
a permutation $\pi$ with descent composition in $\beta$ and a permutation $\sigma$ (disjoint from $\pi$) with descent composition in $\gamma$. Observe that $c_{\beta,\gamma}^{\alpha}=\sum_{L\in\alpha}c_{J,K}^{L}$
for any choice of $J\in\beta$ and $K\in\gamma$, where as before
$c_{J,K}^{L}$ is the number of permutations with descent composition
$L$ that are obtained as a shuffle of a permutation $\pi$ with descent
composition $J$ and a permutation $\sigma$ (disjoint from $\pi$) with descent composition $K$.

Define the linear map $\phi_{\st}\colon\QSym\rightarrow A$ by $\phi_{\st}(F_{L})=u_{\alpha}$
for $L\in\alpha$. Then any $J\in\beta$ and $K\in\gamma$ satisfy
\begin{align*}
\phi_{\st}(F_{J}F_{K}) & =\phi_{\st}\Big(\sum_{L}c_{J,K}^{L}F_{L}\Big)\\
 & =\sum_{L}c_{J,K}^{L}\phi_{\st}(F_{L})\\
 & =\sum_{\alpha}\sum_{L\in\alpha}c_{J,K}^{L}u_{\alpha}\\
 & =\sum_{\alpha}c_{\beta,\gamma}^{\alpha}u_{\alpha}\\
 & =u_{\beta}u_{\gamma}\\
 & =\phi_{\st}(F_{J})\phi_{\st}(F_{K}),
\end{align*}
so $\phi_{\st}$ is a $\mathbb{Q}$-algebra homomorphism, thus completing
one direction of the proof.
The converse follows directly from Theorem \ref{t-quots}.\end{proof}

It is immediate from Theorem \ref{t-scqsym} that when $\st$ is shuffle-compatible,
its shuffle algebra is isomorphic to $\QSym/\ker(\phi_{\st})$.
\begin{cor}
The shuffle algebra of every shuffle-compatible descent statistic
is isomorphic to a quotient algebra of $\QSym$.
\end{cor}

\subsection{Shuffle-compatibility of \texorpdfstring{$\des$}{des} and \texorpdfstring{$(\des,\maj)$}{(des, maj)}}
\label{s-scdesmaj}

We now use Theorem \ref{t-scqsym} to characterize  the shuffle algebras
of the two other shuffle-compatible statistics mentioned in the introduction:
the descent number $\des$ and the pair $(\des,\maj)$. For the latter,
we will actually characterize the shuffle algebra of $(\des,\comaj)$,
but this is sufficient by Theorem \ref{t-esc} since $(\des,\maj)$
and $(\des,\comaj)$ are equivalent statistics. We note that these characterizations can be derived from Propositions 8.3 and 12.6 of Stanley \cite{Stanley1972} in a related way, though we emphasize the connection with quasisymmetric functions.
We will first prove the result for $(\des,\comaj)$ and then derive from it the result for $\des$ using Theorem \ref{t-quots}.

We denote the set of non-negative integers by $\mathbb{N}$.

\begin{thm}[Shuffle-compatibility of $(\des,\comaj)$]
\label{t-descomajsc} 
\leavevmode
\begin{itemize}
\item [\normalfont{(a)}] The ordered pair $(\des,\comaj)$ is shuffle-compatible.
\item [\normalfont{(b)}] The linear map on ${\cal A}_{(\des,\comaj)}$ defined by
\[
[\pi]_{(\des,\comaj)} \mapsto q^{\comaj(\pi)}\qbinom{p-\des(\pi)+\left|\pi\right|-1}{\left|\pi\right|}x^{\left|\pi\right|}
\]
is a $\mathbb{Q}$-algebra isomorphism from ${\cal A}_{(\des,\comaj)}$
to the span of 
\[
\{1\}\bigcup\left\{ q^{k}\qbinom{p-j+n-1}{n}x^{n}\right\} _{n\geq1,\:0\leq j\leq n-1,\:{j+1 \choose 2}\leq k\leq nj-{j+1 \choose 2}},
\]
a subalgebra of $\mathbb{Q}[q,x]^{\mathbb{N}}$, the algebra of functions 
$\mathbb{N}\rightarrow\mathbb{Q}[q,x]$ in the non-negative
integer variable $p$.
\item [\normalfont{(c)}] The linear map on ${\cal A}_{(\des,\comaj)}$ defined by
\[
[\pi]_{(\des,\comaj)}\mapsto\begin{cases}
\displaystyle{\frac{q^{\comaj(\pi)}t^{\des(\pi)+1}}{(1-t)(1-qt)\cdots(1-q^{\left|\pi\right|}t)}x^{\left|\pi\right|}}, & \text{if }\left|\pi\right|\geq1,\\
1/(1-t), & \text{if }\left|\pi\right|=0,
\end{cases}
\]
is a $\mathbb{Q}$-algebra isomorphism from ${\cal A}_{(\des,\comaj)}$
to the span of 
\[
\left\{ \frac{1}{1-t}\right\} \bigcup\left\{ \frac{q^{k}t^{j+1}}{(1-t)(1-qt)\cdots(1-q^{n}t)}x^{n}\right\} _{n\geq1,\:0\leq j\leq n-1\:,{j+1 \choose 2}\leq k\leq nj-{j+1 \choose 2}},
\]
a subalgebra of $\mathbb{Q}[[t*,q]][x]$.
\item [\normalfont{(d)}] For $n\geq1$, the $n$th homogeneous component
of ${\cal A}_{(\des,\comaj)}$ has dimension ${n \choose 3}+n$.
\end{itemize}
\end{thm}
\begin{proof}
We first prove parts (a) and (b). 
For $p$ a positive integer and $f$ a quasisymmetric function, let
\begin{equation*}
\phi^{(p)}_{(\comaj,\des)}(f) = f(x,qx,\dots, q^{p-1}x)
\end{equation*}
and let $\phi^{(0)}_{(\comaj,\des)}(f)$ be the constant term in $f$. 
It is clear that
$\phi^{(p)}_{(\comaj,\des)}$ is a homomorphism from $\QSym$ to $\mathbb{Q}[q,x]$,
so the map that takes $f$ to the function $p\mapsto f(x,qx,\dots, q^{p-1}x)$ is a homomorphism from $\QSym$ to $\mathbb{Q}[q,x]^\mathbb{N}$.

If $L$ is a composition of $n\ge1$, then
\begin{align*}
F_{L}(x,qx,\dots,q^{p-1}x) & =\sum_{\substack{0\leq i_{1}\leq\cdots\leq i_{n}\leq p-1\\
i_{j}<i_{j+1}\,\mathrm{if}\,j\in\Des(L)
}
}q^{i_{1}+\cdots+i_{n}-n}x^n\\
 & =q^{e(L)}\sum_{0\leq r_{1}\leq\cdots\leq r_{n}\leq p-1-\des(L)}q^{r_{1}+\cdots+r_{n}}x^n,
\end{align*}
where 
\[
r_{j}=i_{j}-\left|\left\{\, k : k\in \Des(L)\text{ and } k<j\,\right\} \right|
\]
and 
\[
e(L)=\sum_{j=1}^{n}\left|\left\{\, k: k\in\Des(L)\text{ and } k<j\,\right\} \right|=\comaj(L).
\]
Since
\[
\sum_{0\leq r_{1}\leq\cdots\leq r_{n}\leq p-1-\des(L)}q^{r_{1}+\cdots+r_{n}}=\qbinom{p-\des(L)+n-1}{n}
\]
\cite[Proposition 1.7.3]{Stanley2011}, it follows that
\[
\phi^{(p)}_{(\comaj,\des)}(F_L) =q^{\comaj(L)}\qbinom{p-\des(L)+n-1}{n}x^n,
\]
and for $n=0$ we have $\phi^{(p)}_{(\comaj,\des)}(F_\varnothing)=1$.

Furthermore, it follows from the formula between equations (1.86)
and (1.87) in \cite{Stanley2011} (a form of the $q$-binomial theorem) that 
\begin{align}
\sum_{p=0}^{\infty}\qbinom{p-\des(L)+n-1}{n}t^{p} 
 & =\sum_{p=0}^{\infty}\qbinom{p+n}{n}t^{p+\des(L)+1}\nonumber \\
 & =\frac{t^{\des(L)+1}}{(1-t)(1-qt)\cdots(1-q^{n}t)}.\label{e-qbinom}
\end{align}
Equation \eqref{e-qbinom} implies that the functions $q^{k}\binom{p-j+n-1}{n}_{q}x^{n}$ are linearly
independent as their generating functions are clearly linearly independent. 
Then parts (a) and (b) follow from Theorem \ref{t-scqsym}
and Proposition \ref{p-comajvalues}.

To prove (c), we define the map $\psi\colon\mathbb{Q}[q,x]^{\mathbb{N}}\rightarrow\mathbb{Q}[q,x][[t*]]$ by the formula
\[
\psi(f)=\sum_{p=0}^{\infty}f(p)t^{p}.
\]
Then $\psi$ is clearly an isomorphism and  
by (\ref{e-qbinom}), the images of the basis elements in (b) are those given in (c), which are in $\mathbb{Q}[[t*,q]][x]$.
For $n\geq1$, the number of $(\des,\comaj)$-equivalence classes
for $n$-permutations is 
\[
\sum_{j=0}^{n-1}\left(\left(nj-{j+1 \choose 2}\right)-{j+1 \choose 2}+1\right)=\sum_{j=0}^{n-1}\left(nj-2{j+1 \choose 2}+1\right),
\]
which can be shown to be equal to ${n \choose 3}+n$ by a routine
 argument. This proves (d).
\end{proof}

\begin{thm}[Shuffle-compatibility of the descent number]
\label{t-dessc} \leavevmode
\begin{itemize}
\item [\normalfont{(a)}] The descent number $\des$ is shuffle-compatible.
\item [\normalfont{(b)}] The linear map on ${\cal A}_{\des}$ defined by

\[
[\pi]_{\des} \mapsto {p-\des(\pi)+\left|\pi\right|-1 \choose \left|\pi\right|}x^{\left|\pi\right|}
\]
is a $\mathbb{Q}$-algebra isomorphism from ${\cal A}_{\des}$ to
the span of 
\[
\{1\}\bigcup\left\{ {p-j+n-1 \choose n}x^{n}\right\} _{n\geq1,\:0\leq j\leq n-1},
\]
a subalgebra of $\mathbb{Q}[p,x]$.
\item [\normalfont{(c)}] ${\cal A}_{\des}$ is isomorphic to the span of
\[
\{1\}\cup\{p^{j}x^{n}\}_{n\geq1,\:1\leq j\leq n},
\]
a subalgebra of $\mathbb{Q}[p,x]$.
\item [\normalfont{(d)}] The linear map on ${\cal A}_{\des}$ defined by
\[
[\pi]_{\des}\mapsto\begin{cases}
\displaystyle{\frac{t^{\des(\pi)+1}}{(1-t)^{\left|\pi\right|+1}}x^{\left|\pi\right|}}, & \text{if }\left|\pi\right|\geq1,\\
1/(1-t), & \text{if }\left|\pi\right|=0,
\end{cases}
\]
is a $\mathbb{Q}$-algebra isomorphism from ${\cal A}_{\des}$ to
the span of 
\[
\left\{ \frac{1}{1-t}\right\} \bigcup\left\{ \frac{t^{j+1}}{(1-t)^{n+1}}x^{n}\right\} _{n\geq1,\:0\leq j\leq n-1},
\]
a subalgebra of $\mathbb{Q}[[t*]][x]$.
\item [\normalfont{(e)}] For $n\geq1$, the $n$th homogeneous component
of ${\cal A}_{\des}$ has dimension $n$.
\end{itemize}
\end{thm}

\begin{proof}
Applying Theorem \ref{t-quots} to Theorem \ref{t-descomajsc} with the homomorphism that takes $q$ to~1, together with the observation that polynomial functions in characteristic zero may be identified with polynomials, yields (a), (b), and (d). Parts (c) and (e) follow easily from (b).
\end{proof}

\subsection{Shuffle-compatibility of the peak set and peak number}

In \cite{Stembridge1997}, Stembridge defined a subalgebra $\Pi$
of $\QSym$ called the ``algebra of peaks'' using enriched $P$-partitions,
a variant of Stanley's $P$-partitions. Here, we observe that Stembridge's
algebra $\Pi$ is isomorphic to the shuffle algebra ${\cal A}_{\Pk}$
of the peak set $\Pk$, thus showing that $\Pk$ is shuffle-compatible,
and we use further results of Stembridge on enriched $P$-partitions
to show that the peak number $\pk$ is shuffle-compatible and to characterize
its shuffle algebra.

An enriched $P$-partition is a map defined for a poset $P$, but
for our purposes, we only need to consider the case where $P$ is
a chain. Then the notion of an enriched $P$-partition is equivalent
to that of an ``enriched $\pi$-partition'' for a permutation $\pi$,
which we define below.\footnote{We note that, in the notation of \cite{Stembridge1997}, we are setting $A=\mathbb{P}$, $\gamma = \pi$, and $P=([n],<)$.}

Let $\mathbb{P}^{\prime}$ denote the set of nonzero integers with
the following total ordering: 
\[
-1\prec+1\prec-2\prec+2\prec-3\prec+3\prec\cdots.
\]
For $\pi\in\mathfrak{P}_{n},$ an \textit{enriched $\pi$-partition
}is a map $f\colon[n]\rightarrow\mathbb{P}^{\prime}$ such that for all
$i<j$ in $[n]$, the following hold:
\begin{enumerate}
\item $f(i)\preceq f(j)$;
\item $f(i)=f(j)>0$ implies $\pi(i)<\pi(j)$;
\item $f(i)=f(j)<0$ implies $\pi(i)>\pi(j)$.
\end{enumerate}
Let ${\cal E}(\pi)$ denote the set of enriched $\pi$-partitions,
and let 
\[
\Gamma(\pi)\coloneqq\sum_{f\in{\cal E}(\pi)}x_{\left|f(1)\right|}x_{\left|f(2)\right|}\cdots x_{\left|f(n)\right|}
\]
be the generating function for enriched $\pi$-partitions in which both
$k$ and $-k$ receive the same weight $x_{k}$. For example, let
$\pi=3125674$. Then the map $f$ given by $f(1)=-1$, $f(2)=-1$,
$f(3)=-3$, $f(4)=3$, $f(5)=3$, $f(6)=-7$, $f(7)=9$ is an enriched
$\pi$-partition, which contributes $x_{1}^{2}x_{3}^{3}x_{7}x_{9}$
to $\Gamma(\pi)$.

It is clear that $\Gamma(\pi)$ is a quasisymmetric function homogeneous
of degree $n$ which depends only on the descent set of $\pi$, but
a stronger statement is true: $\Gamma(\pi)$ depends only on the peak
set of $\pi$ \cite[Proposition 2.2]{Stembridge1997}. Hence, it makes
sense to define the quasisymmetric function 
\[
K_{n,\Lambda}\coloneqq\Gamma(\pi)
\]
where $\pi$ is any $n$-permutation with $\Pk(\pi)=\Lambda$. These
peak quasisymmetric functions $K_{n,\Lambda}$ are linearly independent
over $\mathbb{Q}$ \cite[Theorem 3.1(a)]{Stembridge1997}.

Let $F_{n}$ be the $n$th Fibonacci number defined by $F_{1}=F_{2}=1$
and $F_{n}=F_{n-1}+F_{n-2}$ for $n\geq3$. It is easy to see that,
for $n\geq1$, there are exactly $F_{n}$ peak sets among all $n$-permutations,
so the $\mathbb{Q}$-vector space $\Pi_{n}$ spanned by the $K_{n,\Lambda}$
has dimension $F_{n}$ with basis elements corresponding to peak sets
of $n$-permutations. The peak quasisymmetric functions $K_{n,\Lambda}$
multiply by the rule 
\begin{equation}
K_{m,\Pk(\pi)}K_{n,\Pk(\sigma)}=\sum_{\tau\in S(\pi,\sigma)}K_{m+n,\Pk(\tau)}\label{e-pkmult}
\end{equation}
\cite[Equation (3.1)]{Stembridge1997}, so $\Pi\coloneqq\bigoplus_{n=0}^{\infty}\Pi_{n}$
is a $\mathbb{Q}$-algebra, the \textit{algebra of peaks}. Then the
shuffle-compatibility of $\Pk$ and our characterization of the shuffle
algebra ${\cal A}_{\Pk}$ is immediate from (\ref{e-pkmult}).
\begin{thm}[Shuffle-compatibility of the peak set]
\leavevmode
\begin{itemize}
\item [\normalfont{(a)}] The peak set $\Pk$ is shuffle-compatible. 
\item [\normalfont{(b)}] The linear map on ${\cal A}_{\Pk}$ defined by 
\[
[\pi]_{\Pk}\mapsto K_{\left|\pi\right|,\Pk(\pi)}
\]
is a $\mathbb{Q}$-algebra isomorphism from ${\cal A}_{\Pk}$ to $\Pi$.
\end{itemize}

\end{thm}
By Corollary \ref{c-rcsc}, the valley set $\Val$ is also shuffle-compatible
and ${\cal A}_{\Val}$ is isomorphic to $\Pi$. 
Note that (\ref{e-pkmult}) implies that the map $F_{L}\mapsto K_{n,\Pk(L)}$
is a $\mathbb{Q}$-algebra homomorphism from $\QSym$ to itself, 
a fact that we shall use in the proof of the next theorem, which is the
analogous result for the peak number (and by Lemma \ref{l-pkvalruns}, Corollary \ref{c-rcsc}, and Theorem \ref{t-esc},
the valley number and exterior peak number as well).

\begin{thm}[Shuffle-compatibility of the peak number]
\label{t-pksc} \leavevmode
\begin{itemize}
\item [\normalfont{(a)}] The peak number $\pk$ is shuffle-compatible.
\item [\normalfont{(b)}] The linear map on ${\cal A}_{\pk}$ defined by
\[
[\pi]_{\pk}\mapsto\begin{cases}
\displaystyle{\frac{2^{2\pk(\pi)+1}t^{\pk(\pi)+1}(1+t)^{\left|\pi\right|-2\pk(\pi)-1}}{(1-t)^{\left|\pi\right|+1}}x^{\left|\pi\right|}}, & \text{if }\left|\pi\right|\geq1,\\
1/(1-t), & \text{if }\left|\pi\right|=0,
\end{cases}
\]
is a $\mathbb{Q}$-algebra isomorphism from ${\cal A}_{\pk}$ to the
span of 
\[
\left\{ \frac{1}{1-t}\right\} \bigcup\left\{ \frac{2^{2j+1}t^{j+1}(1+t)^{n-2j-1}}{(1-t)^{n+1}}x^{n}\right\} _{n\geq1,\:0\leq j\leq\left\lfloor \frac{n-1}{2}\right\rfloor },
\]
a subalgebra of $\mathbb{Q}[[t*]][x]$.
\item [\normalfont{(c)}] The $\pk$ shuffle algebra ${\cal A}_{\pk}$ is isomorphic to the span of 
\[
\{1\}\cup\{p^{j}x^{n}\}_{n\geq1,\:1\leq j\leq n,\: j\equiv n\,(\mathrm{mod}\,2)},
\]
a subalgebra of $\mathbb{Q}[p,x]$.
\item [\normalfont{(d)}] For $n\geq1$, the $n$th homogeneous component
of ${\cal A}_{\pk}$ has dimension $\left\lfloor (n+1)/2\right\rfloor $.
\end{itemize}

\end{thm}
The proof below implies parts (a), (b), and (d). We postpone the proof of part (c) until Section \ref{s-altdes}.
\begin{proof}

For a quasisymmetric function $f$, let $f(1^{k})$ denote $f$ evaluated at $x_i=1$ for $1\leq i\leq k$ and $x_i=0$ for $i>k$.
Define $\phi_{\pk}\colon\QSym\rightarrow\mathbb{Q}[[t*]][x]$ by
the formula
\[
\phi_{\pk}(F_{L})=\sum_{k=0}^{\infty}K_{n,\Pk(L)}(1^{k})t^{k}x^{n}
\]
for $L\vDash n>0$ and $\phi_{\pk}(F_\varnothing)=1/(1-t$). Then $\phi_{\pk}$
is the composition of the map $F_{L}\mapsto K_{n,\Pk(L)}$ with the
map $f\mapsto\sum_{k=0}^{\infty}f(1^{k})t^{k}x^{n}$ (where $f$ is homogeneous of degree $n$); since both of
these maps are $\mathbb{Q}$-algebra homomorphisms, it follows that
$\phi_{\pk}$ is a $\mathbb{Q}$-algebra homomorphism as well.

Stembridge \cite[Theorem 4.1]{Stembridge1997}
showed that 
\[
\sum_{k=0}^{\infty}K_{n,\Pk(L)}(1^{k})t^{k}=\frac{2^{2\pk(L)+1}t^{\pk(L)+1}(1+t)^{n-2\pk(L)-1}}{(1-t)^{n+1}},
\]
so in fact 
\[
\phi_{\pk}(F_{L})=\frac{2^{2\pk(L)+1}t^{\pk(L)+1}(1+t)^{n-2\pk(L)-1}}{(1-t)^{n+1}}x^{n}.
\]
We know from Proposition \ref{p-pkdesvalues} that for an $n$-permutation
$\pi$, the possible values of $\pk(\pi)$ range from 0 to $\left\lfloor (n-1)/2\right\rfloor $.
Since the elements $2^{2j+1}t^{j+1}(1+t)^{n-2j-1}x^{n}/(1-t)^{n+1}$
are linearly independent, the result follows from Theorem \ref{t-scqsym}.
\end{proof}

An alternative proof of Theorem \ref{t-pksc} can be given using Theorems \ref{t-quots} and \ref{t-pkdessc}.

\subsection{Shuffle-compatibility of the left peak set and left peak number}

Motivated by Stembridge's theory of enriched $P$-partitions and the
study of peak algebras \cite{Nyman2003}, Petersen \cite{Petersen2006,Petersen2007}
defined another variant of $P$-partitions called ``left enriched
$P$-partitions'' that tells a parallel story for left peaks. 

As before, we restrict our attention to when $P$ is a chain. Let
$\mathbb{P}^{(\ell)}$ denote the set of integers with the following
total ordering: 
\[
0\prec-1\prec+1\prec-2\prec+2\prec-3\prec+3\prec\cdots.
\]
Then for $\pi\in\mathfrak{P}_{n},$ a \textit{left enriched $\pi$-partition
}is a map $f\colon[n]\rightarrow\mathbb{P}^{(\ell)}$ such that for all
$i<j$ in $[n]$, the following hold:
\begin{enumerate}
\item $f(i)\preceq f(j)$;
\item $f(i)=f(j)\geq0$ implies $\pi(i)<\pi(j)$;
\item $f(i)=f(j)<0$ implies $\pi(i)>\pi(j)$.
\end{enumerate}
Let ${\cal E}^{(\ell)}(\pi)$ denote the set of left enriched $\pi$-partitions,
and let 
\[
\Gamma^{(\ell)}(\pi)\coloneqq\sum_{f\in{\cal E}^{(\ell)}(\pi)}x_{\left|f(1)\right|}x_{\left|f(2)\right|}\cdots x_{\left|f(n)\right|}.
\]
Just as the generating function $\Gamma(\pi)$ for enriched $\pi$-partitions
depends only on the peak set of $\pi$, Petersen proved that $\Gamma^{(\ell)}(\pi)$
depends only on the left peak set \cite[Corollary 6.5]{Petersen2007},
so we can define 
\[
K_{n,\Lambda}^{(\ell)}\coloneqq\Gamma^{(\ell)}(\pi)
\]
for any $\pi\in\mathfrak{P}_{n}$ with $\Lpk(\pi)=\Lambda$. Unlike
the peak functions $K_{n,\Lambda}$, the $K_{n,\Lambda}^{(\ell)}$
are not quasisymmetric functions but rather type B quasisymmetric
functions.\footnote{We omit the definition of a type B quasisymmetric function, as they
play no further role in this paper, but we refer the reader to \cite{Chow2001}.}

Petersen briefly mentions that the span of the left peak functions
$K_{n,\Lambda}^{(\ell)}$ forms a graded subalgebra $\Pi^{(\ell)}$
of the algebra of type B quasisymmetric functions, called the \textit{algebra
of left peaks} \cite[p. 604]{Petersen2007}.\footnote{Petersen actually calls this algebra the ``left algebra of peaks'',
but the ``algebra of left peaks'' seems to us a more natural name.} The $n$th homogeneous component of $\Pi^{(\ell)}$ has dimension
$F_{n+1}$, which is easily seen to be the number of left peak sets
among $n$-permutations. He does not explicitly state a multiplication
rule for the $K_{n,\Lambda}^{(\ell)}$, but it follows from the fundamental
lemma of left enriched $P$-partitions \cite[Lemma 4.2]{Petersen2007}
that the multiplication is given by
\[
K_{m,\Lpk(\pi)}^{(\ell)}K_{n,\Lpk(\sigma)}^{(\ell)}=\sum_{\tau\in S(\pi,\sigma)}K_{m+n,\Lpk(\tau)}^{(\ell)},
\]
which implies the shuffle-compatibility of the left peak set (and by Corollary \ref{c-rcsc},
the right peak set as well).
\begin{thm}[Shuffle-compatibility of the left peak set]
\label{t-Lpksc} \leavevmode
\begin{itemize}
\item [\normalfont{(a)}] The left peak set $\Lpk$ is shuffle-compatible. 
\item [\normalfont{(b)}] The linear map on ${\cal A}_{\Lpk}$ defined by 
\[
[\pi]_{\Lpk}\mapsto K_{\left|\pi\right|,\Lpk(\pi)}^{(\ell)}
\]
is a $\mathbb{Q}$-algebra isomorphism from ${\cal A}_{\Lpk}$ to
$\Pi^{(\ell)}$.
\end{itemize}
\end{thm}

Although Petersen was the first to explicitly construct the algebra of left peaks, Theorem \ref{t-Lpksc} also follows from the work of Aguiar, Bergeron, and Nyman, who constructed the coalgebra dual to the algebra of left peaks \cite[Proposition 8.3 and Remark 8.7.3]{Aguiar2004}. We will extensively study coalgebras dual to shuffle algebras in Section \ref{s-section5}.

Petersen's work can also be used (in conjunction with Proposition
\ref{p-lpkdesvalues} and Theorem \ref{t-scqsym}) to prove the shuffle-compatibility
of the left peak number. The proof is similar to the proof
of Theorem \ref{t-pksc}, but we use the identity 
\[
\sum_{p=0}^{\infty}K_{n,\Lpk(L)}^{(\ell)}(1^{p})t^{p}=\frac{2^{2\lpk(L)}t^{\lpk(L)}(1+t)^{n-2\lpk(L)}}{(1-t)^{n+1}}
\]
\cite[Theorem 4.6]{Petersen2007}. 
Alternatively, Theorems \ref{t-quots} and \ref{t-lpkdessc} can be used to produce a different proof.

\begin{thm}[Shuffle-compatibility of the left peak number]
\label{t-lpksc} \leavevmode
\begin{itemize}
\item [\normalfont{(a)}] The left peak number $\lpk$ is shuffle-compatible.
\item [\normalfont{(b)}] The linear map on ${\cal A}_{\lpk}$ defined by
\[
[\pi]_{\lpk}\mapsto\begin{cases}
{\displaystyle \frac{2^{2\lpk(\pi)}t^{\lpk(\pi)}(1+t)^{\left|\pi\right|-2\lpk(\pi)}}{(1-t)^{\left|\pi\right|+1}}x^{\left|\pi\right|}}, & \text{if }\left|\pi\right|\geq1,\\
1/(1-t), & \text{if }\left|\pi\right|=0,
\end{cases}
\]
is a $\mathbb{Q}$-algebra isomorphism from ${\cal A}_{\lpk}$ to
the span of 
\[
\left\{ \frac{1}{1-t}\right\} \bigcup\left\{ \frac{2^{2j}t^{j}(1+t)^{n-2j}}{(1-t)^{n+1}}x^{n}\right\} _{n\geq1,\:0\leq j\leq\left\lfloor n/2\right\rfloor },
\]
a subalgebra of $\mathbb{Q}[[t*]][x]$.
\item [\normalfont{(c)}] The $n$th homogeneous component of ${\cal A}_{\lpk}$
has dimension $\left\lfloor n/2\right\rfloor +1$.
\end{itemize}

\end{thm}

By Theorem \ref{c-rcsc}, the right peak number (or, the number of long runs; see Lemma \ref{l-pkvalruns} (d)) is also shuffle-compatible.

\section{Noncommutative symmetric functions and shuffle-compatibility}
\label{s-section5}

\subsection{Algebras, coalgebras, and graded duals}

In this section, we introduce another criterion for shuffle-compatibility
that will be in a sense ``dual'' to the criterion in Theorem \ref{t-scqsym}.
For this, we shall need the notion dual to an algebra, which requires
the following equivalent definition of an algebra.

Let $R$ be a commutative ring. An $R$-\textit{algebra} $A$ is an $R$-module
with an $R$-linear map $\mu\colon A\otimes A\rightarrow A$ such that the
following diagram commutes:
\noindent \begin{center}
$\begin{CD}
A \otimes A \otimes A   @>\id \otimes \mu >>   A \otimes A\\
@V \mu \otimes \id VV                         @VV \mu V\\
A \otimes A             @>> \phantom{bb} \mu \phantom{bb} >             A
\end{CD}$
\par\end{center}

\noindent The map $\mu$ is called a \textit{multiplication}.%
\footnote{The multiplication map $\mu$ satisfies $\mu(a\otimes b)=ab$ under
the original definition of an algebra; from this, it is clear why
$\mu$ is called ``multiplication''.%
}

The notion dual to an algebra is a coalgebra, defined as follows.
An $R$-\textit{coalgebra} $C$ is an $R$-module with an $R$-linear
map $\Delta\colon C\rightarrow C\otimes C$ such that the following diagram
commutes:
\begin{center}
$\begin{CD}
C \otimes C \otimes C   @<\id \otimes \Delta <<   C \otimes C\\
@A \Delta \otimes \id AA                         @AA \Delta A\\
C \otimes C             @<< \phantom{bb} \Delta \phantom{bb} <             C
\end{CD}$
\par\end{center}

\noindent Observe that this diagram is essentially the diagram in
the definition of an algebra, but with arrows reversed. The map $\Delta$
is called a \textit{comultiplication}.%
\footnote{\noindent Typically, the definition of an algebra requires an additional
linear map called a ``unit'' which satisfies a certain commutative
diagram, and the definition of a coalgebra requires the dual concept
of a ``counit'', but these will not be necessary for our work.}

If an $R$-module $A$ is simultaneously an $R$-algebra and an $R$-coalgebra
such that its comultiplication map is an $R$-algebra homomorphism,
then we call $A$ an $R$-\textit{bialgebra}.

Suppose now that $R$ is a field and that $V=\bigoplus_{n\geq0}V_{n}$
is a graded $R$-vector space of finite type, that is, each component
$V_{n}$ is finite-dimensional. Let $V^{o}$ denote the \textit{graded
dual} $V^{o}\coloneqq\bigoplus_{n\geq0}V_{n}^{*}$, which is contained
inside the dual space $V^{*}$ of $V$. We say that a linear map 
$\phi\colon V\rightarrow W$ 
is \textit{graded} if, for every $n\geq0$, $\phi(V_{n})$ is contained
inside the $n$th homogeneous component of $W$. Every graded
linear map $\phi\colon V\rightarrow W$ induces a graded linear map $\phi^{o}\colon W^{o}\rightarrow V^{o}$
given by 
\[
\phi^{o}(f)(v)=f(\phi(v))
\]
for $f\in W^{o}$ and $v\in V$. In particular, if $A$ is a graded
$R$-algebra---meaning that its vector space and multiplication are
graded---and is of finite type, then by reversing the arrows in the
commutative diagram, we see that $A^{o}$ has the structure of
a graded $R$-coalgebra. In fact, if $A$ has basis $\{a_{i}\}$ with
structure constants $\{c_{j,k}^{i}\}$, i.e., 
\[
a_{j}a_{k}=\sum_{i}c_{j,k}^{i}a_{i},
\]
then the $\{c_{j,k}^{i}\}$ are also the structure constants for the comultiplication
of the dual basis $\{f_{i}\}$ in $A^{o}$:
\[
\Delta(f_{i})=\sum_{j,k}c_{j,k}^{i}f_{j}\otimes f_{k}.
\]
Similarly, the graded dual of a graded $R$-coalgebra is a graded
$R$-algebra, with the same correspondence of structure constants.
If $\phi$ is an $R$-algebra homomorphism, then $\phi^{o}$ is an
$R$-coalgebra homomorphism, and vice versa.

\subsection{Noncommutative symmetric functions}

The graded dual of $\QSym$ is the coalgebra of noncommutative symmetric
functions, which also has an algebra structure. We begin by defining
the algebra of noncommutative symmetric functions before introducing
the comultiplication.

Let $\mathbb{Q}\langle\langle X_{1},X_{2},\dots\rangle\rangle$ be
the $\mathbb{Q}$-algebra of formal power series in countably many
noncommuting variables $X_{1},X_{2},\dots$. Consider the elements
\[
\mathbf{h}_{n}\coloneqq\sum_{i_{1}\leq\cdots\leq i_{n}}X_{i_{1}}X_{i_{2}}\cdots X_{i_{n}}
\]
of $\mathbb{Q}\langle\langle X_{1},X_{2},\dots\rangle\rangle$, with 
$\mathbf{h}_{0}=1$, which
are noncommutative versions of the complete symmetric functions $h_{n}$.
Note that $\mathbf{h}_{n}$ is the noncommutative generating function
for weakly increasing words of length $n$ on the alphabet $\mathbb{P}$
of positive integers.
For example, the weakly increasing word $13449$ is encoded by $X_{1}X_{3}X_{4}^{2}X_{9}$,
which appears as a term in $\mathbf{h}_{5}$. Given a composition
$L=(L_{1},\dots,L_{k})$, we let 
\begin{equation}
\label{e-hL-def}
\mathbf{h}_{L}\coloneqq\mathbf{h}_{L_{1}}\cdots\mathbf{h}_{L_{k}}.
\end{equation}
Equivalently, 
\[
\mathbf{h}_{L}=\sum_{i_{1},\dots,i_{n}}X_{i_{1}}X_{i_{2}}\cdots X_{i_{n}}
\]
where the sum is over all $i_{1},\dots,i_{n}$ satisfying 
\[
\underset{L_{1}}{\underbrace{i_{1}\leq\cdots\leq i_{L_{1}}}},\underset{L_{2}}{\underbrace{i_{L_{1}+1}\leq\cdots\leq i_{L_{1}+L_{2}}}},\dots,\underset{L_{k}}{\underbrace{i_{L_{1}+\cdots+L_{k-1}+1}\leq\cdots\leq i_{n}}},
\]
so $\mathbf{h}_{L}$ is the noncommutative generating function for
words in $\mathbb{P}$ whose descent set is contained in $\Des(L)$. 

Let $\mathbf{Sym}_{n}$ denote the vector space spanned by $\{\mathbf{h}_{L}\}_{L\vDash n}$,
and let $\mathbf{Sym}\coloneqq\bigoplus_{n=0}^{\infty}\mathbf{Sym}_{n}$.
Then $\mathbf{Sym}$ is a graded $\mathbb{Q}$-algebra called the
\textit{algebra of noncommutative symmetric functions} with coefficients
in $\mathbb{Q}$, a subalgebra of $\mathbb{Q}\langle\langle X_{1},X_{2},\dots\rangle\rangle$.
The study of $\mathbf{Sym}$ was initiated in \cite{ncsf1}, although
noncommutative symmetric functions have appeared implicitly in earlier 
work, including the first author's Ph.D. thesis \cite{gessel-thesis}.
Also see \cite{Gessel2014,Zhuang2016,Zhuang2016a} for a series of
recent papers by the present authors on the subject of permutation
enumeration in which $\mathbf{Sym}$ plays a role.

In the following sections, we will work with noncommutative symmetric functions with coefficients in either the ring $\Q[x,y]$ of polynomials in $x$ and $y$ with rational coefficients or the ring $\Qtxy$ of polynomials in $x$ and $y$ with coefficients in the ring of formal power series in $t$ in which multiplication is the Hadamard product in $t$ but ordinary multiplication in $x$ and~$y$.
We will also need to use formal sums of noncommutative symmetric functions of unbounded degree with these coefficient rings, for example,  $\sum_{n=0}^\infty \h_n x^n$.  
We will use the notation $\Symxy$ for the algebra of \ncsfs\ of unbounded degree with coefficients in $\Q[x,y]$ and $\Symtxy$ for \ncsfs\ with coefficients in $\Qtxy$. 

For a composition $L=(L_{1},\dots,L_{k})$, we  define 
\[
\mathbf{r}_{L}\coloneqq\sum_{i_{1},\dots,i_{n}}X_{i_{1}}X_{i_{2}}\cdots X_{i_{n}}
\]
where the sum is over all $i_{1},\dots,i_{n}$ satisfying 
\[
\underset{L_{1}}{\underbrace{i_{1}\leq\cdots\leq i_{L_{1}}}}>\underset{L_{2}}{\underbrace{i_{L_{1}+1}\leq\cdots\leq i_{L_{1}+L_{2}}}}>\cdots>\underset{L_{k}}{\underbrace{i_{L_{1}+\cdots+L_{k-1}+1}\leq\cdots\leq i_{n}}}.
\]
Then $\mathbf{r}_{L}$ is the noncommutative generating function
for words on the alphabet $\mathbb{P}$ with descent composition $L$.

Note that 
\begin{equation}
\mathbf{h}_{L}=\sum_{\substack{\Des(K)\subseteq\Des(L)\\
\left|K\right|=\left|L\right|
}
}\mathbf{r}_{K},\label{e-hitor}
\end{equation}
so by inclusion-exclusion, 
\begin{equation}
\mathbf{r}_{L}=\sum_{\substack{\Des(K)\subseteq\Des(L)\\
\left|K\right|=\left|L\right|
}
}(-1)^{l(L)-l(K)}\mathbf{h}_{K}\label{e-ritoh}
\end{equation}
where $l(L)$ denotes the number of parts of the composition $L$.
Hence the $\mathbf{r}_{L}$ are noncommutative symmetric functions,
and are in fact noncommutative versions of the ribbon skew Schur functions
$r_{L}$. 

Since $\mathbf{r}_{L}$ and $\mathbf{r}_{M}$ have no terms in common
for $L\neq M$, it is clear that $\{\mathbf{r}_{L}\}_{L\vDash n}$
is linearly independent. From (\ref{e-hitor}), we see that $\{\mathbf{r}_{L}\}_{L\vDash n}$
spans $\mathbf{Sym}_{n}$, so $\{\mathbf{r}_{L}\}_{L\vDash n}$ is
a basis for $\mathbf{Sym}_{n}$. Because $\{\mathbf{h}_{L}\}_{L\vDash n}$
spans $\mathbf{Sym}_{n}$ and has the same cardinality as $\{\mathbf{r}_{L}\}_{L\vDash n}$,
we conclude that $\{\mathbf{h}_{L}\}_{L\vDash n}$ is also a basis
for $\mathbf{Sym}_{n}$.

Let us also consider the noncommutative generating function 
\[
\mathbf{e}_{n}\coloneqq\sum_{i_{1}>\cdots>i_{n}}X_{i_{1}}X_{i_{2}}\cdots X_{i_{n}}
\]
for decreasing words of length $n$ on the alphabet $\mathbb{P}$.
Then $\mathbf{e}_n$ is a noncommutative version of the elementary symmetric function $e_n$, and $\mathbf{e}_n\in \Sym_n$ since $\mathbf{e}_n=\r_{(1^n)}$.

Let
\[
\mathbf{h}(x)\coloneqq\sum_{n=0}^{\infty}\mathbf{h}_{n}x^{n}
\]
be the generating function for the noncommutative complete symmetric
functions $\mathbf{h}_{n}$, where $x$ commutes with all the variables $X_i$, and let 
\[
\mathbf{e}(x)\coloneqq\sum_{n=0}^{\infty}\mathbf{e}_{n}x^{n}
\]
be the generating function for the $\mathbf{e}_{n}$. Then 
\begin{equation}
\mathbf{e}(x)=\mathbf{h}(-x)^{-1},\label{e-ehr}
\end{equation}
which is a consequence of the infinite product formulas
\begin{equation*}
\h(x) =(1-X_1x)^{-1}(1-X_2x)^{-1}\cdots
\text{\quad and \quad} 
\mathbf{e}(x) = \cdots (1+X_2x)(1+X_1x)
\end{equation*}
(cf. \cite[p.~38]{gessel-thesis} or \cite[Section 7.3]{ncsf1}).

The algebra $\mathbf{Sym}$ can be given a coalgebra structure by
defining the comultiplication $\Delta:\mathbf{Sym}\rightarrow\mathbf{Sym}\otimes\mathbf{Sym}$
by 
\begin{equation}
\label{e-hcomult}
\Delta\mathbf{h}_{n}=\sum_{i=0}^{n}\mathbf{h}_{i}\otimes\mathbf{h}_{n-i}
\end{equation}
and extending by the rule 
\begin{equation}
\Delta(fg)=(\Delta f)(\Delta g).\label{e-deltahomo}
\end{equation}
Since $\Delta$ is an algebra homomorphism, $\mathbf{Sym}$ is a bialgebra.%
\footnote{In fact, both $\mathbf{Sym}$ and $\QSym$ are Hopf algebras (see \cite{Grinberg2014} for a definition) and the duality between $\mathbf{Sym}$ and $\QSym$ given in the next
theorem is in fact a Hopf algebra duality. However, we will not need
the antipode in this paper, nor will we be concerned with the coalgebra structure of $\QSym$.}
The comultiplication $\Delta$ extends naturally to $\Symxy$ and $\Symtxy$ (but note that now tensor products are over the coefficient ring).

Next, we show that the graded dual of the algebra $\QSym$ is the coalgebra $\mathbf{Sym}$; cf.~\cite[Theorem 6.1]{ncsf1} or \cite[Section 5.3]{Grinberg2014}. We may extend the definition of $\h_L$ to weak compositions $L$ by \eqref{e-hL-def}, so that if $L$ is a weak composition then $\h_L = \h_{L'}$ where $L'$ is the composition obtained from $L$ by removing all zero parts. Recall that, as defined in Section \ref{s-bijproof}, weak compositions are added componentwise.

\begin{lem}
\label{l-hcoprod}
Let $L$ be a composition. Then $\Delta \h_L = \sum_{J,K} \h_J \otimes\h_K$,
where the sum is over all pairs of weak compositions $J$ and $K$ with the same number of parts such that $J+K = L$.
\end{lem}
\begin{proof}
This follows easily from the fact that $\Delta \h_{(L_1,\dots, L_m)} = 
\Delta \h_{L_1}\cdots \Delta \h_{L_m}$ together with \eqref{e-hcomult}.
\end{proof} 

\begin{thm}
\label{t-duality} The graded dual of the algebra $\QSym$ of quasisymmetric
functions is isomorphic to the coalgebra $\mathbf{Sym}$ of noncommutative
symmetric functions. In particular, the 
monomial basis $\{M_{L}\}$
of\/ $\QSym$ is dual to the complete basis $\{\mathbf{h}_{L}\}$ of\/ $\mathbf{Sym}$
and the fundamental basis $\{F_{L}\}$
of\/ $\QSym$ is dual to the ribbon basis $\{\mathbf{r}_{L}\}$ of\/ $\mathbf{Sym}$.
\end{thm}

\begin{proof}
We first consider the product of two monomial quasisymmetric functions.
Define coefficients $b^{L}_{J,K}$ by
\begin{equation}
\label{e-Mproduct}
M_J M_K = \sum_L b^{L}_{J,K}M_L. 
\end{equation}
It is easy to see that $b^L_{J,K}$ is the number of pairs of weak compositions $(J',K')$ with the same number of parts such that $J'$ is obtained from $J$ by inserting zeros, $K'$ is obtained from $K$ by inserting zeros, and $J'+K' = L$. 

Lemma \ref{l-hcoprod} implies that
\begin{equation*}
\Delta \h_L = \sum_{J,K} b_{J,K}^L \h_J\otimes \h_K,
\end{equation*}
where the coefficients $b_{J,K}^L$ are the same as those in 
equation \eqref{e-Mproduct}. Thus 
$\{M_L\}_{L\vDash n}$ and $\{\h_L\}_{L\vDash n}$ are dual bases for  $\QSym_n$ and 
$\mathbf{Sym}_n$.

We may define a pairing between $\QSym$ and $\mathbf{Sym}$ by 
\begin{equation*}
\pair{M_{K}}{\h_{L}}=\delta_{K,L}=
\begin{cases}
1, & \text{if }K=L,\\
0, & \text{otherwise.}
\end{cases}
\end{equation*}
Then 
\begin{align*}
\pair{F_K}{\r_L} &= \biggl\langle\,\sum_{\Des(I)\supseteq\Des(K)}M_I,
   \sum_{\Des(J)\subseteq \Des(L)}(-1)^{l(L) - l(J)}\h_J\biggr\rangle\\
   &=\sum_{\substack{\Des(J)\supseteq\Des(K)\\  
      \Des(J)\subseteq\Des(L)}}(-1)^{l(L) - l(J)}=\delta_{K,L},
\end{align*}
and this implies that $\{F_L\}$ and $\{\r_L\}$ are dual bases.
\end{proof}

\subsection{Monoidlike elements}

We call an element $f$ of a bialgebra \emph{monoidlike} if $\Delta f = f\otimes f$. It is straightforward to show that the product of two monoidlike elements
is monoidlike and that the inverse of a monoidlike element, if it
exists, is monoidlike.%
\footnote{%
A monoidlike element $f$ of a bialgebra is called \emph{grouplike} if $\varepsilon(f)$ is the identity element of the coefficient ring, where $\varepsilon$ is the counit. In our bialgebras, the counit is the coefficient of $\h_0$,  the identity element of $\Q$ or $\Q[x,y]$ is 1, and the identity element of $\Qtxy$ is $(1-t)^{-1}=\sum_{k=0}^\infty t^k$. Nearly all of our monoidlike elements are actually grouplike, but exceptions occur in Corollary \ref{c-monoidlike2}.}

\begin{lem}
\label{l-hemonoidlike}
$\mathbf{h}(x)$, $\mathbf{e}(x)$, and $\mathbf{e}(xy)$ are monoidlike in $\Symxy$.
\end{lem}
\begin{proof}
We have 
\begin{align*}
\Delta\mathbf{h}(x) & =\sum_{n=0}^{\infty}\Delta\mathbf{h}_{n}x^{n}\\
 & =\sum_{n=0}^{\infty}\sum_{i+j=n}(\mathbf{h}_{i}\otimes\mathbf{h}_{j})x^{n}\\
 & =\sum_{n=0}^{\infty}\sum_{i+j=n}\mathbf{h}_{i}x^{i}\otimes\mathbf{h}_{j}x^{j}\\
 & =\sum_{i,j=0}^{\infty}\mathbf{h}_{i}x^{i}\otimes\mathbf{h}_{j}x^{j}\\
 & =\Big(\sum_{i=0}^{\infty}\mathbf{h}_{i}x^{i}\Big)\otimes\Big(\sum_{j=0}^{\infty}\mathbf{h}_{j}x^{j}\Big),
\end{align*}
so $\mathbf{h}(x)$ is monoidlike. Since $\mathbf{e}(x)=\mathbf{h}(-x)^{-1}$,
this implies that $\mathbf{e}(x)$ and
$\mathbf{e}(xy)$ are monoidlike.
\end{proof}

\begin{lem}
\label{l-monoidlike2}
Let $f=\sum_{n=0}^\infty a_n t^n$ be an element of $\Symtxy$ where each $a_n$ is an element of $\Symxy$. Then $f$ is monoidlike in $\Symtxy$ if and only if each $a_n$ is monoidlike in $\Symxy$.
\end{lem}
\begin{proof}
We have 
\begin{align*}
f\otimes f&= \sum_{m,n=0}^\infty a_mt^m \otimes a_n t^n\\
  &= \sum_{m,n=0}^\infty (a_m \otimes a_n) (t^m * t^n)\\
  &=\sum_{n=0}^\infty (a_n \otimes a_n) t^n
\shortintertext{and}
\Delta f&=\sum_{n=0}^\infty \Delta a_n t^n.
\end{align*}
Thus $\Delta f = f\otimes f$ if and only if $\Delta a_n = a_n \otimes a_n$ for each $n$.
\end{proof}

The next result follows immediately from Lemma \ref{l-monoidlike2}.
\begin{cor}
\label{c-monoidlike2}
Suppose that $f$ is monoidlike in $\Symxy$. Then $(1-tf)^{-1}$, $(1-t^2f)^{-1}$, and $1+tf$
are monoidlike in $\Symtxy$. 
\end{cor}

\subsection{Implications of duality to shuffle-compatibility}

Let $\st$ be a descent statistic. For each $\st$-equivalence
class $\alpha$ of compositions, let 
\[
\mathbf{r}_{\alpha}^{\st} \coloneqq \sum_{L\in\alpha}\mathbf{r}_{L}.
\]
We call the noncommutative symmetric functions $\mathbf{r}_{\alpha}^{\st}$ $\st$-\textit{ribbons}.

The following is the dual version of Theorem \ref{t-scqsym}.
\begin{thm}
\label{t-rib} 
A descent statistic $\st$ is shuffle-compatible if and only if for every $\st$-equivalence
class $\alpha$ of compositions, there exist constants $c_{\beta,\gamma}^{\alpha}$
for which 
\[
\Delta\mathbf{r}_{\alpha}^{\st}=\sum_{\beta,\gamma}c_{\beta,\gamma}^{\alpha}\mathbf{r}_{\beta}^{\st}\otimes\mathbf{r}_{\gamma}^{\st};
\]
that is, the $\st$-ribbons $\mathbf{r}_{\alpha}^{\st}$ span a subcoalgebra of $\mathbf{Sym}$.
In this case, the $c_{\beta,\gamma}^{\alpha}$ are the structure constants
for ${\cal A}_{\st}$.
\end{thm}
\begin{proof}

By Theorem \ref{t-duality}, we have a pairing between 
quasisymmetric functions and noncommutative symmetric functions
for which
\[
\pair{F_{L}}{\mathbf{r}_{J}}=\begin{cases}
1, & \mbox{if }L=J,\\
0, & \mbox{otherwise.}
\end{cases}
\]

Suppose that the $\st$-ribbons $\mathbf{r}_{\alpha}^{\st}$ span a subcoalgebra
of $\mathbf{Sym}$ with structure constants $c_{\beta,\gamma}^{\alpha}$. 
Let $D$ be the subcoalgebra spanned by the $\mathbf{r}_{\alpha}^{\st}$
and let $i\colon D\rightarrow\mathbf{Sym}$ be the canonical inclusion map, a $\mathbb{Q}$-coalgebra
homomorphism. Then $i$ induces a $\mathbb{Q}$-algebra homomorphism
$i^{o}\colon\QSym\rightarrow D^{o}$ given by 
\begin{align*}
i^{o}(F_{L})(\mathbf{r}_{\alpha}^{\st}) & =\pair{F_{L}}{i(\mathbf{r}_{\alpha}^{\st})}\\
 & =\pair{F_{L}}{\mathbf{r}_{\alpha}^{\st}}\\
 & =\begin{cases}
1, & \mbox{if }L\in\alpha,\\
0, & \mbox{otherwise.}
\end{cases}
\end{align*}
Observe that $i^{o}(F_{L})=i^{o}(F_{J})$ whenever $L$ and $J$ belong
to the same $\st$-equivalence class. Hence, we can define $f_{\alpha} \coloneqq i^{o}(F_{L})$
for $L\in\alpha$. Then $\{f_{\alpha}\}$ is the basis of $D^{o}$
dual to $\{\mathbf{r}_{\alpha}^{\st}\}$, so 
\[
f_{\beta}f_{\gamma}=\sum_{\alpha}c_{\beta,\gamma}^{\alpha}f_{\alpha}.
\]
By Theorem \ref{t-scqsym}, $\st$ is shuffle-compatible with shuffle
algebra isomorphic to $D^{o}$. We omit the  proof of the reverse implication, as it
is similar; we begin with a quotient algebra of $\QSym$ and
then show that its basis elements are dual to the $\st$-ribbons $\mathbf{r}_{\alpha}^{\st}$.
\end{proof}
While Theorem \ref{t-scqsym} tells us that we can prove the shuffle-compatibility
of a descent statistic by constructing suitable quotients of $\QSym$,
Theorem \ref{t-rib} tells us that we could, alternatively, construct
suitable subcoalgebras of $\mathbf{Sym}$, and this is what we will do in Sections \ref{s-pkdes} to \ref{s-udrdes}.
Moreover, because it is straightforward to compute coproducts of noncommutative symmetric functions, Theorem \ref{t-rib} is useful for showing that a descent statistic is not shuffle-compatible and for conjecturing that a statistic is shuffle-compatible, which is not the case for Theorem~\ref{t-scqsym}.

Although Theorem \ref{t-rib} does not give us a way to describe the dual algebra
${\cal A}_{\st}$, we can describe ${\cal A}_{\st}$ explicitly using the following theorem. 
For an $\st$-equivalence class $\alpha$ of compositions, we let $|\alpha|$ be the sum of the parts of any composition $L\in \alpha$.

\begin{thm}
\label{l-monoidlikesc} Let $\st$
be a descent statistic and let $u_{\alpha}\in \Qtxy$
be linearly independent elements \textup{(}\kern -1pt over $\mathbb{Q}$\textup{)} indexed by $\st$-equivalence classes
$\alpha$ of compositions. Suppose that $f=\sum_{\alpha}u_{\alpha}\mathbf{r}_{\alpha}^{\st}$
is monoidlike in $\Symtxy$ and that there exist constants
$c_{\beta,\gamma}^{\alpha}$ such that $u_{\beta}u_{\gamma}=\sum_{\alpha}c_{\beta,\gamma}^{\alpha}u_{\alpha}$
for all $\st$-equivalence classes $\beta$ and~$\gamma$,
where $c_{\beta,\gamma}^{\alpha}=0$ unless
$|\alpha| = |\beta|+|\gamma|$.
 Then $\st$ is
shuffle-compatible and the linear map defined by
\[
[\pi]_{\st}\mapsto u_{\alpha},
\]
where $\Comp(\pi)\in\alpha$, is a $\mathbb{Q}$-algebra isomorphism
from ${\cal A}_{\st}$ to the subalgebra of $\Qtxy$
spanned by the $u_{\alpha}$.
\end{thm}
\begin{proof}
Since $f$ is monoidlike,
we have that 
\begin{align*}
\sum_{\alpha}u_{\alpha}\Delta\mathbf{r}_{\alpha}^{\st}
 & =\Delta f
 =\Big(\sum_{\beta}u_{\beta}\mathbf{r}_{\beta}^{\st}\Big)\otimes\Big(\sum_{\gamma}u_{\gamma}\mathbf{r}_{\gamma}^{\st}\Big)\\
 & =\sum_{\beta,\gamma}u_{\beta}u_{\gamma}\mathbf{r}_{\beta}^{\st}\otimes\mathbf{r}_{\gamma}^{\st}\\
 & =\sum_{\alpha}u_{\alpha}\sum_{\beta,\gamma}c_{\beta,\gamma}^{\alpha}\mathbf{r}_{\beta}^{\st}\otimes\mathbf{r}_{\gamma}^{\st}.
\end{align*}
Extracting the linear combinations of elements of $\Sym_i\otimes \Sym_j$, where $i+j=n$, we obtain
\[
\sum_{|\alpha| = n}u_{\alpha}\Delta\mathbf{r}_{\alpha}^{\st}=\sum_{|\alpha| = n}u_{\alpha}\sum_{\beta,\gamma}c_{\beta,\gamma}^{\alpha}\mathbf{r}_{\beta}^{\st}\otimes\mathbf{r}_{\gamma}^{\st}.
\]
Since these are finite sums, linear independence of the $u_{\alpha}$ implies 
\[
\Delta\mathbf{r}_{\alpha}^{\st}=\sum_{\beta,\gamma}c_{\beta,\gamma}^{\alpha}\mathbf{r}_{\beta}^{\st}\otimes\mathbf{r}_{\gamma}^{\st}
\]
and it follows from Theorem \ref{t-rib} that $\st$ is shuffle-compatible
and that the $c_{\beta,\gamma}^{\alpha}$ are the structure constants
for ${\cal A}_{\st}$. Since 
\[
u_{\beta}u_{\gamma}=\sum_{\alpha}c_{\beta,\gamma}^{\alpha}u_{\alpha}
\]
for all $\st$-equivalence classes $\beta$ and $\gamma$, the map
$[\pi]_{\st}\mapsto u_{\alpha}$ is an algebra homomorphism from ${\cal A}_{\st}$
to the subalgebra of $\Qtxy$ spanned by the $u_{\alpha}$,
and since the $u_{\alpha}$ are linearly independent, this map is
an isomorphism.
\end{proof}

We note that Theorem \ref{l-monoidlikesc} can be generalized to a statement about monoidlike elements of more general graded bialgebras; we stated it only in the special case that we will use.

Unfortunately, in our applications, it is  difficult to show directly
that the desired $u_{\alpha}$ are closed under multiplication. The following variant of Theorem \ref{l-monoidlikesc} uses a change of basis argument to deal with this problem.

\begin{thm}
\label{l-monoidlikesc1} Let $\st$
be a descent statistic and let $u_{\alpha}\in\Qtxy$ 
be linearly independent elements \textup{(}\kern -1pt over $\mathbb{Q}$\textup{)} indexed by $\st$-equivalence classes
$\alpha$ of compositions.
Suppose that  $f=\sum_{\alpha}u_{\alpha}\mathbf{r}_{\alpha}^{\st}$
is monoidlike in $\Symtxy$, where $u_\alpha$ is $x^{|\alpha|}$ times an element of $\Q[[t*]][y]$. Let $\mathbf{s}_{n,\px,\qx}$ be the coefficient of $x^n y^\px t^\qx $ in $\sum_{\alpha}u_{\alpha}\mathbf{r}_{\alpha}^{\st}$ and suppose that $\mathbf{r}_{\alpha}^{\st}\in \Span_{\mathbb{Q}}\{\mathbf{s}_{n,\px,\qx}\}$ for each $\alpha$. Then $\st$ is shuffle-compatible and the linear map defined by
\[
[\pi]_{\st}\mapsto u_{\alpha},
\]
where $\Comp(\pi)\in\alpha$, is a $\mathbb{Q}$-algebra isomorphism
from ${\cal A}_{\st}$ to the subalgebra of $\Qtxy$
spanned by the $u_{\alpha}$.
\end{thm}

\begin{proof}
Equating coefficients of $x^n$ in  
\begin{equation*}
f = \sum_{\alpha}u_{\alpha}\mathbf{r}_{\alpha}^{\st}=\sum_{n,\px,\qx} x^n y^\px t^\qx  \mathbf{s}_{n,\px,\qx}
\end{equation*}
gives
\begin{equation*}
 \sum_{|\alpha|=n}u_{\alpha}\mathbf{r}_{\alpha}^{\st}=x^n\sum_{\px,\qx}y^\px t^\qx \mathbf{s}_{n,\px,\qx}.
\end{equation*}
Since the sum on the left is finite,
this 
shows that $\mathbf{s}_{n,\px,\qx}\in\Span_{\mathbb{Q}}\{\mathbf{r}_{\alpha}^{\st}\}$, so 
$\Span_{\mathbb{Q}}\{\mathbf{r}_{\alpha}^{\st}\}=\Span_{\mathbb{Q}}\{\mathbf{s}_{n,\px,\qx}\}$.

Let $f_\qx$ be the coefficient of $t^\qx$ in $f$. Then since $f$ is monoidlike, $f_{\qx}$ is monoidlike by Lemma \ref{l-monoidlike2}, so 
\begin{align*}
\sum_{n,\px} x^n y^\px \Delta\mathbf{s}_{n,\px,\qx}
   &= \Delta f_{\qx}
   =f_{\qx}\otimes f_{\qx}\\
   &=\Big(\sum_{n_1,\px_1} x^{n_1} y^{\px_1}  \mathbf{s}_{n_1,\px_1,\qx}\Big)
      \otimes \Big(\sum_{n_2,\px_2} x^{n_2} y^{\px_2}  
        \mathbf{s}_{n_2,\px_2,\qx}\Big)\\
   &=\sum_{n_1,\px_1,n_2,\px_2}x^{n_1+n_2}y^{\px_1+\px_2}   
      \mathbf{s}_{n_1,\px_1,\qx}\otimes\mathbf{s}_{n_2,\px_2,\qx}
\end{align*}
Equating coefficients of $x^ny^\px$ shows that $\Span_{\mathbb{Q}}\{\mathbf{s}_{n,\px,\qx}\}$ is a subcoalgebra of $\Sym$ and thus so is $\Span_{\mathbb{Q}}\{\mathbf{r}_{\alpha}^{\st}\}$. As a result, there exist constants $c_{\beta,\gamma}^{\alpha}$ such that 
\[
\Delta\mathbf{r}_{\alpha}^{\st}=\sum_{\beta,\gamma}c_{\beta,\gamma}^{\alpha}\mathbf{r}_{\beta}^{\st}\otimes\mathbf{r}_{\gamma}^{\st},
\]
so it follows from Theorem \ref{t-rib} that $\st$ is shuffle-compatible
and that the $c_{\beta,\gamma}^{\alpha}$ are the structure constants
for ${\cal A}_{\st}$.

Moreover, since $\sum_{\alpha}u_{\alpha}\mathbf{r}_{\alpha}^{\st}$
is monoidlike, we have  
\begin{align*}
\sum_{\beta,\gamma}\sum_{\alpha}u_{\alpha}c_{\beta,\gamma}^{\alpha}\mathbf{r}_{\beta}^{\st}\otimes\mathbf{r}_{\gamma}^{\st} & =\sum_{\alpha}u_{\alpha}\sum_{\beta,\gamma}c_{\beta,\gamma}^{\alpha}\mathbf{r}_{\beta}^{\st}\otimes\mathbf{r}_{\gamma}^{\st}\\
 & =\sum_{\alpha}u_{\alpha}\Delta\mathbf{r}_{\alpha}^{\st}\\
 & =\Delta\Big(\sum_{\alpha}u_{\alpha}\mathbf{r}_{\alpha}^{\st}\Big)\\
 & =\Big(\sum_{\beta}u_{\beta}\mathbf{r}_{\beta}^{\st}\Big)\otimes\Big(\sum_{\gamma}u_{\gamma}\mathbf{r}_{\gamma}^{\st}\Big)\\
 & =\sum_{\beta,\gamma}u_{\beta}u_{\gamma}\mathbf{r}_{\beta}^{\st}\otimes\mathbf{r}_{\gamma}^{\st}.
\end{align*}
Using the linear independence of the 
 $\mathbf{r}_{\beta}^{\st}\otimes\mathbf{r}_{\gamma}^{\st}$ and the fact that for each $i$ and $j$,  $\mathbf{r}_{\beta}^{\st}\otimes\mathbf{r}_{\gamma}^{\st}\in \Sym_i \otimes \Sym_j$ for only finitely many $\beta$ and $\gamma$, 
we may equate
coefficients of $\mathbf{r}_{\beta}^{\st}\otimes\mathbf{r}_{\gamma}^{\st}$
to obtain $u_{\beta}u_{\gamma}=\sum_{\alpha}c_{\beta,\gamma}^{\alpha}u_{\alpha}$.
Thus the map $[\pi]_{\st}\mapsto u_{\alpha}$ is an algebra homomorphism
from ${\cal A}_{\st}$ to the subalgebra of $\mathbb{Q}[[t*]][x,y]$
spanned by the $u_{\alpha}$, and since the $u_{\alpha}$ are linearly
independent, this map is an isomorphism.
\end{proof}

Before applying Theorem \ref{l-monoidlikesc1} to prove new results, let us see how it works in a simpler case, the shuffle-compatibility of the descent number (Theorem \ref{t-dessc}).

We start with the formula
\begin{equation}
\label{e-desh}
(1-t\h(x))^{-1}=\frac{1}{1-t}
 +\sum_{n=1}^{\infty}\sum_{L\vDash n}
 \frac{t^{\des(L)+1}}{(1-t)^{n+1}}x^{n}\mathbf{r}_{L},
\end{equation}
which is the case $y=0$ of Equation (\ref{e-pkdes}) below, but is easily proved directly 
\cite[p.~83, Equation (3)]{gessel-thesis}.
Let $\mathbf{r}_{n,j}^{\des}$, for $n\ge1$, denote the noncommutative symmetric
function $\mathbf{r}_{\alpha}^{\des}$ where $\alpha$ is the
$\des$-equivalence class of compositions corresponding to $n$-permutations
with $j-1$ descents, and let $\r_{0,j}^{\des}=\delta_{0,j}$.
Let
\begin{equation}
\label{e-udes}
u_{n,j}=u_{\alpha}=\frac{t^{j}}{(1-t)^{n+1}}x^{n}
\end{equation}
for $n\ge0$. Then $\sum_\alpha u_\alpha \r_{\alpha}^{\des}$ is equal to \eqref{e-desh}, which is monoidlike in $\Qtx$ by Lemma \ref{l-hemonoidlike} and Corollary \ref{c-monoidlike2}.

With the notation of Theorem \ref{l-monoidlikesc1}, we have for fixed $n\ge1$,
\[
\sum_{\qx=0}^{\infty}t^{\qx}\mathbf{s}_{n,0,\qx}
=\sum_{j=1}^{n}\frac{t^{j}}{(1-t)^{n+1}}\mathbf{r}_{n,j}^{\des},
\]
Multiplying both sides by $(1-t)^{n+1}$ and equating coefficients of powers of $t$ shows that $\r_{n,j}^{\des}\in \Span_{\mathbb{Q}}\{\mathbf{s}_{n,0,\qx}\}$. 
So by Theorem \ref{l-monoidlikesc1}, we obtain part (d) of Theorem \ref{t-dessc}.

\subsection{Shuffle-compatibility of \texorpdfstring{$(\pk,\des)$}{(pk, des)}}
\label{s-pkdes}

In the remainder of Section \ref{s-section5}, we use Theorem \ref{l-monoidlikesc1}
to establish the shuffle-compatibility and describe the shuffle algebras
of the descent statistics $(\pk,\des)$, $(\lpk,\des)$, $(\udr,\des)$, and $\udr$. 
All computations are done
in the algebra $\Symtxy$ of noncommutative symmetric functions with coefficients
in $\Qtxy$.
We start with the shuffle-compatibility of $(\pk,\des)$.
 
\begin{thm}[Shuffle-compatibility of $(\pk,\des)$]
\label{t-pkdessc} \leavevmode
\begin{itemize}
\item [\normalfont{(a)}] The pair $(\pk,\des)$ is shuffle-compatible.
\item [\normalfont{(b)}] The linear map on ${\cal A}_{(\pk,\des)}$ defined by
\begin{multline*}
[\pi]_{(\pk,\des)}\mapsto\\
\begin{cases}
{\displaystyle \frac{t^{\pk(\pi)+1}(y+t)^{\des(\pi)-\pk(\pi)}(1+yt)^{\left|\pi\right|-\pk(\pi)-\des(\pi)-1}(1+y)^{2\pk(\pi)+1}}{(1-t)^{\left|\pi\right|+1}}x^{\left|\pi\right|}}, & \text{if }\left|\pi\right|\geq1,\\
1/(1-t), & \text{if }\left|\pi\right|=0,
\end{cases}
\end{multline*}
is a $\mathbb{Q}$-algebra isomorphism from ${\cal A}_{(\pk,\des)}$
to the span of 
\[
\left\{ \frac{1}{1-t}\right\} \bigcup\left\{ \frac{t^{j+1}(y+t)^{k-j}(1+yt)^{n-j-k-1}(1+y)^{2j+1}}{(1-t)^{n+1}}x^{n}\right\} _{\substack{n\geq1,\hphantom{...................}\\
0\leq j\leq\left\lfloor (n-1)/2\right\rfloor ,\\
j\leq k\leq n-j-1\hphantom{......}
}
},
\]
a subalgebra of $\Qtxy$. 

\item [\normalfont{(c)}] The $(\pk,\des)$ shuffle algebra ${\cal A}_{(\pk,\des)}$ is isomorphic to the span of 
\[
\{1\}\cup\{p^{n-j}(1+y)^{n}(1-y)^{n-2k}x^{n}\}_{n\geq1,\:0\leq j\leq n-1,\:0\leq k\leq\left\lfloor j/2\right\rfloor },
\]
a subalgebra of $\mathbb{Q}[p,x,y]$.

\item [\normalfont{(d)}] For $n\geq1$, the $n$th homogeneous component
of ${\cal A}_{(\pk,\des)}$ has dimension $\left\lfloor (n+1)^{2}/4\right\rfloor $.
\end{itemize}

\end{thm}
We prove here parts (a), (b), and (d). We postpone the proof of part (c) until Section \ref{s-altdes}.

\begin{proof}
By Lemma 4.1 of \cite{Zhuang2016a}, we have the formula
\begin{multline}
\label{e-pkdes}
(1-t\mathbf{e}(xy)\mathbf{h}(x))^{-1}=\frac{1}{1-t}+\\
\sum_{n=1}^{\infty}\sum_{L\vDash n}\frac{t^{\pk(L)+1}(y+t)^{\des(L)-\pk(L)}(1+yt)^{n-\pk(L)-\des(L)-1}(1+y)^{2\pk(L)+1}}{(1-t)^{n+1}}x^{n}\mathbf{r}_{L}.
\end{multline}

Let $\mathbf{r}_{n,j,k}^{(\pk,\des)}$ denote the noncommutative symmetric
function $\mathbf{r}_{\alpha}^{(\pk,\des)}$ where $\alpha$ is the
$(\pk,\des)$-equivalence class of compositions corresponding to $n$-permutations
with $j-1$ peaks and $k-1$ descents. 
By (\ref{e-pkdes}) and Proposition \ref{p-pkdesvalues},
we have 
\begin{align*}
 & (1-t\mathbf{e}(xy)\mathbf{h}(x))^{-1}\\
 & \qquad\qquad=\frac{1}{1-t}+\sum_{n=1}^{\infty}\sum_{j=0}^{\left\lfloor (n-1)/2\right\rfloor }\sum_{k=j}^{n-j-1}\frac{t^{j+1}(y+t)^{k-j}(1+yt)^{n-j-k-1}(1+y)^{2j+1}}{(1-t)^{n+1}}x^{n}\mathbf{r}_{n,j+1,k+1}^{(\pk,\des)}\\
 & \qquad\qquad=\frac{1}{1-t}+\sum_{n=1}^{\infty}\sum_{j=1}^{\left\lfloor (n+1)/2\right\rfloor }\sum_{k=j}^{n-j+1}\frac{t^{j}(y+t)^{k-j}(1+yt)^{n-j-k+1}(1+y)^{2j-1}}{(1-t)^{n+1}}x^{n}\mathbf{r}_{n,j,k}^{(\pk,\des)},
\end{align*}
and this is monoidlike by Lemma \ref{l-hemonoidlike} and 
Corollary \ref{c-monoidlike2}.

Now define $\mathbf{s}_{n,\px,\qx}$ by 
\begin{align*}
\sum_{n,\px,\qx=0}^{\infty}x^{n}y^{\px}t^{\qx}\mathbf{s}_{n,\px,\qx}
 & =(1-t\mathbf{e}(xy)\mathbf{h}(x))^{-1}.
\end{align*}
For fixed $n\geq1,$ we have 
\[
\sum_{\px,\qx=0}^{\infty}y^{\px}t^{\qx}\mathbf{s}_{n,\px,\qx}=\sum_{j=1}^{\left\lfloor (n+1)/2\right\rfloor }\sum_{k=j}^{n-j+1}\frac{t^{j}(y+t)^{k-j}(1+yt)^{n-j-k+1}(1+y)^{2j-1}}{(1-t)^{n+1}}\mathbf{r}_{n,j,k}^{(\pk,\des)}.
\]
This identity can be inverted
to obtain 
\[
\sum_{j=1}^{\left\lfloor (n+1)/2\right\rfloor }\sum_{k=j}^{n-j+1}y^{j}t^{k}\mathbf{r}_{n,j,k}^{(\pk,\des)}=(1+u)\left(\frac{1-v}{1+uv}\right)^{n+1}\sum_{\px,\qx=0}^{\infty}u^{\px}v^{\qx}\mathbf{s}_{n,\px,\qx},
\]
where 
\[
u=\frac{1+t^{2}-2yt-(1-t)\sqrt{(1+t)^{2}-4yt}}{2(1-y)t}
\]
and 
\[
v=\frac{(1+t)^{2}-2yt-(1+t)\sqrt{(1+t)^{2}-4yt}}{2yt},
\]
in the formal power series ring $\Q[[t,y]]$. 
It is easily checked that $u$ and $v$ are both formal power series divisible by $t$, so $(1-v)/(1+uv)$ is a well-defined formal power series in $t$ and $y$.

Equating coefficients of $y^\px  t^\qx $  shows that each $\mathbf{r}_{n,j,k}^{(\pk,\des)}$ is a linear
combination of the $\mathbf{s}_{n,\px,\qx}$. (Since $u$ and $v$ are divisible by $t$, only finitely many terms on the right will contribute a term in $t^\qx $.)
Parts (a) and (b) then follow from Theorem \ref{l-monoidlikesc1}.

By Proposition \ref{p-pkdesvalues}, we know that for $n\geq1$, the number of $(\pk,\des)$-equivalence classes for
$n$-permutations is 
\[
\sum_{j=0}^{\left\lfloor (n-1)/2\right\rfloor }((n-j-1)-j+1)
   =\sum_{j=0}^{\left\lfloor (n-1)/2\right\rfloor }(n-2j),
\]
which is easily shown to be equal to $\left\lfloor (n+1)^{2}/4\right\rfloor$. This proves (d).
\end{proof} 

Note that $(\pk,\des)$ and $(\val,\des)$ are $rc$-equivalent statistics, and that $(\val,\des)$ and $(\epk,\des)$ are equivalent statistics. Thus, by Corollary \ref{c-rcsc} and Theorem \ref{t-esc}, we know that $(\val,\des)$ and $(\epk,\des)$ are also shuffle-compatible and have shuffle algebras isomorphic to $\cal{A}_{(\pk,\des)}$.

\subsection{Shuffle-compatibility of \texorpdfstring{$(\lpk,\des)$}{(lpk, des)}}

We now prove the shuffle-compatibility of $(\lpk,\des)$ and characterize its shuffle algebra.

\begin{thm}[Shuffle-compatibility of $(\lpk,\des)$]
\label{t-lpkdessc} \leavevmode
\begin{itemize}
\item [\normalfont{(a)}] The pair $(\lpk,\des)$ is shuffle-compatible.
\item [\normalfont{(b)}] The linear map on ${\cal A}_{(\lpk,\des)}$ defined by
\begin{multline*}
[\pi]_{(\lpk,\des)}\mapsto\\
\begin{cases}
\displaystyle{\frac{t^{\lpk(\pi)}(y+t)^{\des(\pi)-\lpk(\pi)}(1+yt)^{\left|\pi\right|-\lpk(\pi)-\des(\pi)}(1+y)^{2\lpk(\pi)}}{(1-t)^{\left|\pi\right|+1}}x^{\left|\pi\right|}}, & \text{if }\left|\pi\right|\geq1,\\
1/(1-t), & \text{if }\left|\pi\right|=0,
\end{cases}
\end{multline*}
is a $\mathbb{Q}$-algebra isomorphism from ${\cal A}_{(\lpk,\des)}$
to the span of 
\[
\left\{ \frac{1}{1-t}\right\} \bigcup\left\{ \frac{(1+yt)^{n}}{(1-t)^{n+1}}x^{n}\right\} _{n\geq1}\bigcup\left\{ \frac{t^{j}(y+t)^{k-j}(1+yt)^{n-j-k}(1+y)^{2j}}{(1-t)^{n+1}}x^{n}\right\} _{\substack{n\geq2,\hphantom{............}\\
1\leq j\leq\left\lfloor n/2\right\rfloor ,\\
j\leq k\leq n-j\phantom{...}
}
},
\]
a subalgebra of $\Qtxy$.
\item [\normalfont{(c)}] The $n$th homogeneous component of ${\cal A}_{(\lpk,\des)}$
has dimension $\left\lfloor n^{2}/4\right\rfloor +1$.
\end{itemize}
\end{thm}
\begin{proof}
By Lemma 4.6 of \cite{Zhuang2016a}, we have the formula
\begin{multline*}
\mathbf{h}(x)(1-t\mathbf{e}(xy)\mathbf{h}(x))^{-1}=\frac{1}{1-t}+\\
\sum_{n=1}^{\infty}\sum_{L\vDash n}\frac{t^{\lpk(L)}(y+t)^{\des(L)-\lpk(L)}(1+yt)^{n-\lpk(L)-\des(L)}(1+y)^{2\lpk(L)}}{(1-t)^{n+1}}x^{n}\mathbf{r}_{L}.
\end{multline*}

Let $\mathbf{r}_{n,j,k}^{(\lpk,\des)}$ denote $\mathbf{r}_{\alpha}^{(\lpk,\des)}$
where $\alpha$ is the $(\lpk,\des)$-equivalence class of compositions
corresponding to $n$-permutations with $j$ left peaks and $k$ descents.
Define $\mathbf{s}_{n,\px,\qx}$ by 
\begin{align*}
\sum_{n,\px,\qx=0}^{\infty}x^{n}y^{\px}t^{\qx}\mathbf{s}_{n,\px,\qx}
 & =\mathbf{h}(x)(1-t\mathbf{e}(xy)\mathbf{h}(x))^{-1}.
\end{align*}
Then the proofs for parts (a) and (b) follow
in the same manner as for Theorem \ref{t-pkdessc}, using
Proposition \ref{p-lpkdesvalues} and
Corollary \ref{c-monoidlike2} along the way.

By Proposition \ref{p-lpkdesvalues}, the number of $(\lpk,\des)$-equivalence classes for $n$-permutations
is 
\[
1+\sum_{j=1}^{\left\lfloor n/2\right\rfloor }((n-j)-j+1)=1+\sum_{j=1}^{\left\lfloor n/2\right\rfloor }(n-2j+1),
\]
which is easily shown to be equal to $\left\lfloor n^{2}/4\right \rfloor +1$. This proves (c).
\end{proof}

Although $(\lpk, \des)$ and $(\rpk, \des)$ are not equivalent, $r$-equivalent, $c$-equivalent, or $rc$-equivalent, this argument does show that $(\rpk, \des)$ is shuffle-compatible and has shuffle algebra isomorphic to that of $(\lpk, \des)$ because $(\lpk, \des)$ is $r$-equivalent to $(\rpk, \asc)$---where $\asc$ is the number of ascents%
---and $(\rpk, \asc)$ is equivalent to $(\rpk, \des)$.

\subsection{Shuffle-compatibility of \texorpdfstring{$\udr$}{udr} and \texorpdfstring{$(\udr,\des)$}{(udr,des)}}
\label{s-udrdes}

Finally, we prove our result for the pair $(\udr,\des)$ and derive from it the analogous result for $\udr$, the number of up-down runs.

\begin{thm}[Shuffle-compatibility of $(\udr,\des)$]
\label{t-udsc} \leavevmode
\begin{itemize}
\item [\normalfont{(a)}] The pair $(\udr,\des)$ is shuffle-compatible.
\item [\normalfont{(b)}] The linear map on ${\cal A}_{(\udr,\des)}$ defined by
\begin{equation*}
[\pi]_{(\udr,\des)}\mapsto\\
\begin{cases}
\displaystyle{\frac{N_\pi}{(1-t)(1-t^2)^{\left|\pi\right|}}x^{\left|\pi\right|}}, & \text{if }\left|\pi\right|\geq1,\\
1/(1-t), & \text{if }\left|\pi\right|=0,
\end{cases}
\end{equation*}
where
\def\ucp{\ceil{\udr(\pi)/2}}
\def\ufp{\floor{\udr(\pi)/2}}
\begin{multline*}
N_\pi = t^{\udr(\pi)}(1+y)^{\udr(\pi)-1}(1+yt^{2})^{|\pi|-\des(\pi)-\ucp}
(y+t^{2})^{\des(\pi)-\ufp}\\
\times(1+yt)^{\ucp-\ufp}(y+t)^{1-\ucp+\ufp},
\end{multline*}

is a $\mathbb{Q}$-algebra isomorphism from ${\cal A}_{(\udr,\des)}$
to the span of 
\begin{multline*}
\left\{ \frac{1}{1-t}\right\} 
\bigcup\,
\left\{ \frac{t(1+yt)(1+yt^2)^{n-1}}{(1-t)(1-t^2)^{n}}x^{n}\right\} _{\!n\geq1}
\\
\bigcup\,
\left\{
\frac{t^{j}(1+y)^{j-1}(1+yt^2)^{n-k-\ceil{j/2}}(y+t^2)^{k-\floor{j/2}}S_j}
{(1-t)(1-t^2)^n}x^n
\right\} 
_{\substack{n\geq1,\hfill\\
2\le j\le n ,\hfill\\
\floor{j/2}\leq k\leq n-\ceil{j/2}
}},
\end{multline*}
where $S_j$ is $1+yt$ if $j$ is odd and is $y+t$ if $j$ is even,
a subalgebra of $\Qtxy$.

\item [\normalfont{(c)}] The $n$th homogeneous component of ${\cal A}_{(\udr,\des)}$
has dimension $\binom{n}{2}+1$.
\end{itemize}
\end{thm}

\begin{proof}

By Lemma 4.11 of \cite{Zhuang2016a}, together with Lemma \ref{l-udr} (b) and (c), we have
\begin{equation}
\label{e-ud1}
(1-t^{2}\mathbf{h}(x)\mathbf{e}(xy))^{-1}(1+t\mathbf{h}(x))
=\frac{1}{1-t}+\sum_{n=1}^{\infty}
\sum_{L\vDash n}\frac{N_L}{(1-t)(1-t^{2})^{n}}x^{n}\mathbf{r}_{L}
\end{equation}
where
\begin{multline*}
N_L = t^{\udr(L)}(1+y)^{\udr(L)-1}(1+yt^{2})^{n-\des(L)-\ceil{\udr(L)/2}}
(y+t^{2})^{\des(L)-\floor{\udr(L)/2}}\\
\times(1+yt)^{\ceil{\udr(L)/2} -\floor{\udr(L)/2}}(y+t)^{1-\ceil{\udr(L)/2}+\floor{\udr(L)/2}}.\quad
\end{multline*}
Note that $\ceil{\udr(L)/2} -\floor{\udr(L)/2}$ is 1 if $\udr(L)$ is odd and is 0 if $\udr(L)$ is even. The left-hand side of \eqref{e-ud1} is monoidlike by  
Lemma \ref{l-hemonoidlike} and Corollary \ref{c-monoidlike2}.

Let $\mathbf{r}_{n,j,k}^{(\udr,\des)}$ denote $\mathbf{r}_{\alpha}^{(\udr,\des)}$
where $\alpha$ is the $(\udr,\des)$-equivalence class of compositions
corresponding to $n$-permutations with $j$ up-down runs and $k$ descents.
Then by \eqref{e-ud1} and Proposition \ref{p-udrdesvalues}, we have
\begin{multline}
\label{e-ud2}
(1-t^{2}\mathbf{h}(x)\mathbf{e}(xy))^{-1}(1+t\mathbf{h}(x))
=\frac{1}{1-t}+\sum_{n=1}^{\infty}
\biggl(
\frac{t(1+yt)(1+yt^2)^{n-1}}{(1-t)(1-t^2)^{n}}x^{n}\mathbf{r}_{n,1,0}^{(\udr,\des)}\\
+\sum_{\substack{2\le j\le n\\ \floor{j/2}\le k\le n-\ceil{j/2}}}
\frac{t^{j}(1+y)^{j-1}(1+yt^2)^{n-k-\ceil{j/2}}(y+t^2)^{k-\floor{j/2}}S_j}
{(1-t)(1-t^2)^n} x^n
\mathbf{r}_{n,j,k}^{(\udr,\des)}\biggr)
\end{multline}
with $S_j$ as in the statement of the theorem. 
Define $\mathbf{s}_{n,\px,\qx}$ by 
\begin{equation}
\label{e-s-ud}
\sum_{n,\px,\qx=0}^{\infty}x^{n}y^{\px}t^{\qx}\mathbf{s}_{n,\px,\qx}
   =(1-t^{2}\mathbf{h}(x)\mathbf{e}(xy))^{-1}(1+t\mathbf{h}(x)).
\end{equation}
To prove (a) and (b), as in Theorems \ref{t-pkdessc} and \ref{t-lpkdessc}, it is sufficient to show that each $\mathbf{r}_{n,j,k}^{(\udr,\des)}$ is in the span of the 
$\mathbf{s}_{n,\px,\qx}$. Because of the floor and ceiling functions in \eqref{e-ud2}, we are not able 
to use the generating function inversion method that we used in the proofs of Theorems \ref{t-pkdessc} and \ref{t-lpkdessc}, so we take a different approach.

Expanding the right side of \eqref{e-ud2} and comparing with \eqref{e-s-ud} shows that, for fixed $n$, each $\mathbf{s}_{n,\px,\qx}$ is a linear combination (with integer coefficients) of the $\mathbf{r}_{n,j,k}^{(\udr,\des)}$. 
We will show that these relations can be inverted to express each $\mathbf{r}_{n,j,k}^{(\udr,\des)}$ as a linear combination of the $\mathbf{s}_{n,\px,\qx}$.

We totally order $\N\times \N$ colexicographically, so  $(\px_1,\qx_1)\le (\px_2,\qx_2)$ if and only if $\qx_1<\qx_2$ or $\qx_1=\qx_2$ and $\px_1\le \px_2$. 
We shall show that for each $j$ and $k$, there exist $p$ and $q$ such that $\mathbf{r}_{n,j,k}^{(\udr,\des)}$ appears with coefficient 1 in $\mathbf{s}_{n,\px,\qx}$ and if
$\mathbf{r}_{n,j',k'}^{(\udr,\des)}$ appears in $\mathbf{s}_{n,\px,\qx}$ then $(k', j')\le (k,j)$. This will imply, by induction, that  $\mathbf{r}_{n,j,k}^{(\udr,\des)}$  is in $\Span_{\mathbb{Q}}\{\mathbf{s}_{n,\px,\qx}\}$.

With this total order,  the monomial $y^{\px}t^{\qx}$  with minimal $(p,q)$ that appears in the coefficient of $x^n \mathbf{r}_{n,j,k}^{(\udr,\des)}$ on the right side of \eqref{e-ud2} is easily seen to be $y^{k_j}t^j$ (with coefficient 1),
where $k_j$ is $k-\floor{j/2}+1$ if $j$ is even and is $k-\floor{j/2}$ if $j$ is odd. In other words, $\mathbf{s}_{n,\px,\qx}$ does not contain any $\mathbf{r}_{n,j,k}^{(\udr,\des)}$ for which $(p,q)<(k_j,j)$.
Replacing $\px$ and $\qx$ with $k_j$ and $j$, and replacing $k$ and $j$ with $k'$ and $j'$, we have that  
\begin{equation*}
\mathbf{s}_{n,k_j,j}=\mathbf{r}_{n,j,k}^{(\udr,\des)}+\sum_{j'\!,\, k'}c_{j'\!,\, k'} \mathbf{r}_{n,j',k'}^{(\udr,\des)}
\end{equation*}
where  $c_{j'\!,\, k'}=0$ unless $(k'_{j'}, j') < (k_j,j)$.
It is easy to see that $(k'_{j'}, j') < (k_j,j)$ implies $(k', j') < (k,j)$, so 
we have
\begin{equation*}
\mathbf{s}_{n,k_j,j}=\mathbf{r}_{n,j,k}^{(\udr,\des)}+\sum_{(k'\!,\, j')<(k,j)}c_{j'\!,\, k'} \mathbf{r}_{n,j',k'}^{(\udr,\des)}
\end{equation*}
and this completes the proof of (b).

By Proposition \ref{p-udrdesvalues}, the number of $(\udr,\des)$-equivalence classes for $n$-permutations
is 
\[
1+\sum_{j=2}^{n}(n-\floor{j/2} - \ceil{j/2}+1)
 = 1+\sum_{j=2}^n (n-j+1) = 1+\binom{n}{2}.
\]
This proves part (c).
\end{proof}

We know from Lemma \ref{l-udr} that $\udr$ and $(\lpk,\val)$ are equivalent statistics, from Lemma \ref{l-pkvalruns} (d) that $\val$ is equivalent to $\epk$, and from Proposition \ref{p-lpkpk} that $(\lpk,\val)$ is $rc$-equivalent to $(\lpk,\pk)$. It follows that $(\udr,\des)$ is equivalent to $(\lpk,\val,\des)$ and $(\lpk,\epk,\des)$, and is $rc$-equivalent to $(\lpk,\pk,\des)$. Thus, by Theorem \ref{t-esc} and Corollary \ref{c-rcsc}, the statistics $(\lpk,\val,\des)$, $(\lpk,\epk,\des)$, and $(\lpk,\pk,\des)$ are all shuffle-compatible and have shuffle algebras isomorphic to ${\cal A}_{(\udr,\des)}$.

\begin{thm}[Shuffle-compatibility of the number of up-down runs]
\label{t-udrsc} \leavevmode
\begin{itemize}
\item [\normalfont{(a)}] The number of up-down runs $\udr$ is shuffle-compatible.
\item [\normalfont{(b)}] The linear map on ${\cal A}_{\udr}$ defined by
\begin{align*}
[\pi]_{\udr}\mapsto & \begin{cases}
{\displaystyle \frac{2^{\udr(\pi)-1}t^{\udr(\pi)}(1+t^{2})^{\left|\pi\right|-\udr(\pi)}}{(1-t)^{2}(1-t^{2})^{\left|\pi\right|-1}}x^{\left|\pi\right|}}, & \text{if }\left|\pi\right|\geq1,\\
1/(1-t), & \text{if }\left|\pi\right|=0,
\end{cases}
\end{align*}
is a $\mathbb{Q}$-algebra isomorphism from ${\cal A}_{\udr}$ to
the span of 
\[
\left\{ \frac{1}{1-t}\right\} \bigcup\left\{ \frac{2^{j-1}t^{j}(1+t^{2})^{n-j}}{(1-t)^{2}(1-t^{2})^{n-1}}x^{n}\right\} _{n\geq1,\:1\leq j\leq n},
\]
a subalgebra of $\mathbb{Q}[[t*]][x]$.
\item [\normalfont{(c)}] For $n\geq1$, the $n$th homogeneous component
of ${\cal A}_{\udr}$ has dimension $n$.
\end{itemize}

\end{thm}
\begin{proof}
Let $\phi$ be the homomorphism from $\Qtxy$ to $\Qtx$ obtained by setting $y$ to 1. 
It is easy to check that $\phi$ takes the image of $[\pi]_{(\udr,\des)}$ as described in Theorem \ref{t-udsc} (b) to the image of $[\pi]_{\udr}$ as given in (b). Then (a) and (b) follow from Theorem \ref{t-quots}. Part (c) follows from Proposition \ref{p-udrdesvalues}.
\end{proof}

Since $\udr$ and $(\lpk,\val)$ are equivalent statistics, 
 $(\lpk,\val)$ is shuffle-compatible and ${\cal A}_{(\lpk,\val)}$ is isomorphic to ${\cal A}_{\udr}$.
Furthermore, since $(\lpk,\val)$ is $rc$-equivalent to $(\lpk,\pk)$, we have also proven the shuffle-compatibility of $(\lpk,\pk)$ and characterized the shuffle algebra ${\cal A}_{(\lpk,\pk)}$. Similar reasoning implies that $(\lpk,\epk)$, $(\rpk,\val)$, $(\rpk,\pk)$, $(\rpk,\epk)$, $(\lr,\val)$, $(\lr,\pk)$, and $(\lr,\epk)$ are shuffle-compatible and that their shuffle algebras are all isomorphic to ${\cal A}_{\udr}$.

\section{Miscellany}
\label{s-section6}

\subsection{An alternate description of the \texorpdfstring{$\pk$}{pk} and \texorpdfstring{$(\pk,\des)$}{(pk, des)} shuffle algebras}
\label{s-altdes}

In Section 5.5, we showed that the $(\pk,\des)$ shuffle algebra
${\cal A}_{(\pk,\des)}$ is isomorphic to the span of 
\[
\left\{ \frac{1}{1-t}\right\} \bigcup\left\{ \frac{t^{j+1}(y+t)^{k-j}(1+yt)^{n-j-k-1}(1+y)^{2j+1}}{(1-t)^{n+1}}x^{n}\right\} _{\substack{n\geq1,\hphantom{...................}\\
0\leq j\leq\left\lfloor (n-1)/2\right\rfloor ,\\
j\leq k\leq n-j-1\hphantom{......}
}
}
\]
where the multiplication is the Hadamard product in $t$. Let 
\[
P_{n,j,k}(y,t)\coloneqq t^{j+1}(y+t)^{k-j}(1+yt)^{n-j-k-1}(1+y)^{2j+1}
\]
for $n\geq1$, $0\leq j\leq\left\lfloor (n-1)/2\right\rfloor $, and
$j\leq k\leq n-j-1$. Then by \cite[Corollary 4.3.1]{Stanley2011},
we can write 
\[
\frac{P_{n,j,k}(y,t)}{(1-t)^{n+1}}=\sum_{p=1}^{\infty}R_{n,j,k}(p,y)t^{p}
\]
where $R_{n,j,k}(p,y)$ is a polynomial in $p$ of degree at most
$n$, with coefficients that are polynomials in $y$. In this section,
we give a simple description of the span of the polynomials $R_{n,j,k}(p,y)$,
which yields an alternate characterization of the $(\pk,\des)$ shuffle
algebra that was stated in part (c) of Theorem \ref{t-pkdessc}.
Similarly, a simple description of the span of the polynomials $R_{n,j,k}(p,1)$ yields an alternate
characterization of the $\pk$ shuffle algebra, which is part (c) of Theorem \ref{t-pksc}.

It is simpler to work with the following transformations of the polynomials
$R_{n,j,k}(p,y)$ and $P_{n,j,k}(y,t)$; let
\[
Q_{n,j,k}(p,z)\coloneqq(1-z)^{n}R_{n,j,k}\left(p,\frac{1+z}{1-z}\right)
\]
and let
\begin{align*}
A_{n,j,k}(t,z) & \coloneqq(1-z)^{n}P_{n,j,k}\left(\frac{1+z}{1-z},t\right)\\
 & =(1-z)^{n}t^{j+1}\left(\frac{1+z}{1-z}+t\right)^{k-j}\left(1+\frac{1+z}{1-z}t\right)^{n-j-k-1}\left(1+\frac{1+z}{1-z}\right)^{2j+1}\\
 & =2^{2j+1}t^{j+1}(1+t+z(1-t))^{k-j}(1+t-z(1-t))^{n-j-k-1},
\end{align*}
so that 
\begin{equation}
\label{e-A(t,z)}
\frac{A_{n,j,k}(t,z)}{(1-t)^{n+1}}=\sum_{p=1}^{\infty}Q_{n,j,k}(p,z)t^{p}.
\end{equation}
Also, define $\bar{A}_{n,j,k}(t,z)$ by 
\begin{equation}
\label{e-barA(t,z)}
\frac{\bar{A}_{n,j,k}(t,z)}{(1-t)^{n+1}}=\sum_{p=0}^{\infty}Q_{n,j,k}(-p,z)t^{p}.
\end{equation}

\begin{lem}
\label{l-Qnct} Each $Q_{n,j,k}(p,z)$, as a polynomial in $p$, has
no constant term.\end{lem}
\begin{proof}
By \cite[Proposition 4.2.3]{Stanley2011},  from \eqref{e-A(t,z)} and \eqref{e-barA(t,z)} follows the equality
of rational functions 
\[
\frac{\bar{A}_{n,j,k}(t,z)}{(1-t)^{n+1}}=-\frac{A_{n,j,k}(1/t,z)}{(1-(1/t))^{n+1}},
\]
which implies 
\begin{align*}
\bar{A}_{n,j,k}(t,z) & =(-1)^{n}t^{n+1}A_{n,j,k}(1/t,z)\\
 & =(-1)^{n}2^{2j+1}t^{n+1}\left(\frac{1}{t}\right)^{j+1}\left(1+\frac{1}{t}+z\left(1-\frac{1}{t}\right)\right)^{k-j}\left(1+\frac{1}{t}-z\left(1-\frac{1}{t}\right)\right)^{n-j-k-1}\\
 & =(-1)^{n}2^{2j+1}t^{j+1}(1+t-z(1-t))^{k-j}(1+t+z(1-t))^{n-j-k+1}.
\end{align*}
Evaluating at $t=0$ yields $\bar{A}_{n,j,k}(0,z)=0$, so by \eqref{e-barA(t,z)}, $Q_{n,j,k}(0,z)=0$.
\end{proof}
\begin{lem}
Let $n\geq1$. Then the polynomials $Q_{n,j,k}(p,z)$ for $0\leq j\leq\left\lfloor (n-1)/2\right\rfloor $
and $j\leq k\leq n-j-1$ are linearly independent.\end{lem}
\begin{proof}
It is easy to see that the polynomials $P_{n,j,k}(y,t)$ are linearly
independent, and that a linear dependence relation for the polynomials
$Q_{n,j,k}(p,z)$ would imply a linear dependence relation for the polynomials
$P_{n,j,k}(y,t)$.
\end{proof}
Essentially the same argument can be used to show that the polynomials
$R_{n,j,k}(p,y)$ are also linearly independent.
\begin{thm}
Let $n\geq1$. Then 
\[
\Span_{\mathbb{Q}}\{Q_{n,j,k}(p,z)\}_{\substack{0\leq j\leq\left\lfloor (n-1)/2\right\rfloor ,\\
j\leq k\leq n-j-1\hphantom{......}
}
}=\Span_{\mathbb{Q}}\{p^{n-a}z^{a-2b}\}_{\substack{0\leq a\leq n-1,\\
0\leq b\leq\left\lfloor a/2\right\rfloor 
}
}
\]
\end{thm}
\begin{proof}
First, we show that each $Q_{n,j,k}(p,z)$ can be written as a linear
combination of the polynomials $p^{n-a}z^{a-2b}$. Note that 
\begin{align*}
\sum_{p=1}^{\infty}Q_{n,j,k}(p,z)t^{p} & =\frac{2^{2j+1}t^{j+1}(1+t+z(1-t))^{k-j}(1+t-z(1-t))^{n-j-k-1}}{(1-t)^{n+1}}
\end{align*}
is a linear combination of terms of the form 
\[
\frac{z^{l}t^{q}(1-t)^{l}}{(1-t)^{n+1}} = \frac{t^{q}z^{l}}{(1-t)^{n-l+1}} = \sum_{p=0}^{\infty}z^{l}{n-l+p-q \choose n-l}t^{p}
\]
where $0\leq l\leq n-2j-1$ and $j+1\leq q\leq n-j-l$. Moreover,
${n-l+p-q \choose n-l}$ is a polynomial in $p$ of degree $n-l$, so
it is a linear combination of $1,p,p^{2},\dots,p^{n-l}$. This shows
that each $Q_{n,j,k}(p,z)$ is a linear combination of terms of the
form $p^{n-a}z^{l}$ with $n-a\leq n-l$, or equivalently, $l\leq a$,
and $a\leq n-1$
by Lemma \ref{l-Qnct}. We set $c=a-l$, so
that $p^{n-a}z^{l}=p^{n-a}z^{a-c}$. It remains to show that $c$
must be even.

Observe that $(-p)^{n-a}(-z)^{a-c}=(-1)^{n}(-1)^{c}p^{n-a}z^{a-c}$.
Thus, it suffices to show that $Q_{n,j,k}(p,z)=(-1)^{n}Q_{n,j,k}(-p,-z)$.
Recall that 
\[
\bar{A}_{n,j,k}(t,z)=(-1)^{n}t^{n+1}A_{n,j,k}(1/t,z),
\]
so that 
\[
\bar{A}_{n,j,k}(t,-z)=(-1)^{n}t^{n+1}A_{n,j,k}(1/t,-z).
\]
Since 
\begin{align*}
A_{n,j,k}(t,z) & =2^{2j+1}t^{j+1}(t+1-z(t-1))^{k-j}(t+1+z(t-1))^{n-j-k+1}\\
 & =2^{2j+1}t^{n+1}\left(\frac{1}{t}\right)^{j+1}\left(1+\frac{1}{t}-z\left(1-\frac{1}{t}\right)\right)^{k-j}\left(1+\frac{1}{t}+z\left(1-\frac{1}{t}\right)\right)^{n-j-k-1}\\
 & =t^{n+1}A_{n,j,k}(1/t,-z),
\end{align*}
we have 
\begin{align*}
\sum_{p=1}^{\infty}(-1)^{n}Q_{n,j,k}(p,z)t^{p} & =\frac{(-1)^{n}A(t,z)}{(1-t)^{n+1}}\\
 & =\frac{(-1)^{n}t^{n+1}A(1/t,-z)}{(1-t)^{n+1}}\\
 & =\frac{\bar{A}(t,-z)}{(1-t)^{n+1}}\\
 & =\sum_{p=1}^{\infty}Q_{n,j,k}(-p,-z)t^{p}.
\end{align*}
Therefore, $Q_{n,j,k}(p,z)=(-1)^{n}Q_{n,j,k}(-p,-z)$, so each $Q_{n,j,k}(p,z)$
is a linear combination of the polynomials $p^{n-a}z^{a-2b}$.

Since we know that the polynomials $Q_{n,j,k}(m,z)$ are linearly
independent, it suffices to show that the two sets of polynomials
have the same cardinality. The restrictions $0\leq a\leq n-1$ and
$0\leq b\leq\left\lfloor a/2\right\rfloor $ can be reformulated as
$0\leq b\leq\left\lfloor (n-1)/2\right\rfloor $ and $2b\leq a\leq n-1$;
the restriction on $b$ matches the condition on $j$, and the number
of possible values of $a$ for a fixed $b$ is equal to the number
of possible values of $k$ for a fixed $j$. Hence, the two sets are
equinumerous and thus their spans are equal.
\end{proof}
We are now ready to prove our alternate characterization of ${\cal A}_{(\pk,\des)}$ and of ${\cal A}_{\pk}$.

\begin{proof}[Proof of Theorem \ref{t-pkdessc} $\mathrm{(}c\mathrm{)}$]
In this proof, we identify ${\cal A}_{(\pk,\des)}$ with its characterization given in part (b) of Theorem \ref{t-pkdessc}.

Let $\psi\colon{\cal A}_{(\pk,\des)}\rightarrow\mathbb{Q}[p,x,y]$ be the
linear map defined by 
\[
\psi\Big(\sum_{p=1}^{\infty}R_{n,j,k}(p,y)t^{p}x^{n}\Big)=R_{n,j,k}(p,y)x^{n}
\]
and $\psi(1/(1-t))=1$. With the usual multiplication of $\mathbb{Q}[p,x,y]$,
it is easy to see that $\psi$ is an algebra homomorphism, and thus
restricts to an algebra isomorphism from ${\cal A}_{(\pk,\des)}$
to the subalgebra of $\mathbb{Q}[p,x,y]$ spanned by the $R_{n,j,k}(p,y)x^{n}$.

Observe that 
\[
\Span_{\mathbb{Q}}\{R_{n,j,k}(p,y)\}_{\substack{0\leq j\leq\left\lfloor (n-1)/2\right\rfloor ,\\
j\leq k\leq n-j-1\hphantom{......}
}
}=\Span_{\mathbb{Q}}\{p^{n-a}(1+y)^{n}(1-y)^{a-2b}\}_{\substack{0\leq a\leq n-1,\\
0\leq b\leq\left\lfloor a/2\right\rfloor 
}
};
\]
this is immediate from the previous theorem and applying the inverse
transformation: dividing by $(1-z)^{n}$ and setting $z=(y-1)/(1+y)$. Then the result follows.
\end{proof}

\begin{proof}[Proof of Theorem \ref{t-pksc} $\mathrm{(}c\mathrm{)}$]
First note that setting $y=1$ in the basis for ${\cal A}_{(\pk,\des)}$ given by part (b) of Theorem \ref{t-pkdessc} gives the basis for ${\cal A}_{\pk}$ described in part (b) of Theorem \ref{t-pksc}. Thus setting $y=1$ in the basis for ${\cal A}_{(\pk,\des)}$ given by part (c) of Theorem \ref{t-pkdessc} will give a spanning set for ${\cal A}_{\pk}$.

The only polynomials $p^{n-a}(1+y)^{n}(1-y)^{a-2b}x^{n}$ that are nonzero after setting $y=1$
are those for which $a=2b$, yielding the polynomials $2^{n}p^{n-2b}x^{n}$
for $0\leq2b\leq n-1$. The span of these polynomials is equal to
the span of $p^{j}x^{n}$ for $1\leq j\leq n$ with $j$ having the same parity as $n$.
\end{proof}
We note that part (c) of Theorem \ref{t-pksc} can also be proven using Stembridge's self-reciprocity property for enriched order polynomials \cite[Proposition 4.2]{Stembridge1997}.

Unfortunately, we were unable to use the approach in this section
to give an alternate characterization of any of the shuffle algebras ${\cal A}_{\lpk}$, ${\cal A}_{(\lpk,\des)}$, ${\cal A}_{\udr}$, or ${\cal A}_{(\udr,\des)}$.

\subsection{Non-shuffle-compatible permutation statistics}

Although many well-known descent statistics have been shown to be
shuffle-compatible, there are many descent statistics that are
not shuffle-compatible. Here we list some of them.
\begin{thm}
\label{t-pairfalse} The set $\Pk\cup\Val$ and the tuples $(\pk,\val)$, $(\pk,\val,\des)$,  $(\Pk,\des)$, $(\Pk,\val)$, $(\Pk,\val,\des)$, $(\Pk,\Val)$,  $(\Lpk,\des)$, $(\Lpk,\val,\des)$, and $(\Epk,\des)$ are not shuffle-compatible.
\end{thm}

Recall that a birun of a permutation is a maximal monotone consecutive subsequence, and that $\br(\pi)$ is the number of biruns
of $\pi$. The number of biruns is not shuffle-compatible, and the only joint statistics involving $\br$ that we have found that seem to be shuffle-compatible are $(\Lpk, \br)$ and $(\Epk, \br)$; however, these are easily shown to be equivalent to $\Epk$, which is shuffle-compatible (see the discussion following Conjecture \ref{cj-sc}).
\begin{thm}
The number of biruns $\br$ and the tuples $(\br,\des)$, $(\br,\maj)$, $(\br,\des,\maj)$,
$(\br,\pk)$, $(\br,\pk,\des)$, $(\br,\lpk)$,
$(\br,\lpk,\des)$, and $(\Pk,\br)$ are not shuffle-compatible.
\end{thm}

Although $(\des,\maj)$ is shuffle-compatible, we have not found any other shuffle-compatible joint statistics involving the major index.

\begin{thm}
The tuples $(\pk,\maj)$, $(\pk,\des,\maj)$, $(\lpk,\maj)$, $(\lpk,\des,\maj)$, $(\Pk, \maj)$, $(\Lpk,\maj)$,
$(\udr,\maj)$,  $(\udr,\des,\maj)$, and $(\lir,\maj)$ are not shuffle-compatible.
\end{thm}

In addition to the descent statistics examined in this paper, we mention
that there are two additional families of descent statistics, one based
on the classical notion of double descents and one based on the more
recent notion of alternating descents. We say that $i$ (where $2\leq i\leq n-1$)
is a \textit{double descent} of $\pi\in\mathfrak{P}_{n}$ if $\pi_{i-1}>\pi_{i}>\pi_{i+1}$;
then we can define the double descent set and double descent number---as well as variations of these such as the left double descent set and left double descent number---in the obvious
way. We say that $i\in[n-1]$ is an \textit{alternating descent} if
$i$ is an even ascent or an odd descent; then we can define the alternating descent set,
alternating descent number, and alternating major index in the obvious
way. Alternating descents were introduced by Chebikin \cite{chebikin}
and have been more recently studied by Remmel \cite{remmel} and by the
present authors \cite{Gessel2014}.

Aside from the alternating descent set\textemdash which is equivalent
to the descent set\textemdash none of these statistics mentioned above
are shuffle-compatible. Among joint statistics that involve one or
more of these statistics, we have not found any that seem to be shuffle-compatible
(other than a few that are equivalent to statistics that we know to
be shuffle-compatible).

Lastly, among permutation statistics that are not descent statistics, we have not found any that seem to be shuffle-compatible.

\subsection{Open problems and conjectures}

To conclude this paper, we state a couple permutation statistics 
that we conjecture to be shuffle-compatible based on empirical evidence, and a few more general open problems and conjectures on the topic of shuffle-compatibility.

\begin{conjecture}
\label{cj-sc} The tuples $(\udr,\pk)$ and $(\udr,\pk,\des)$ are shuffle-compatible.
\end{conjecture}

In a preliminary version of this paper, we included as part of Conjecture \ref{cj-sc} the conjectured shuffle-compatibility of the exterior peak set $\Epk$ and the tuples $(\Pk,\val,\des)$, $(\Pk,\udr)$, $(\Lpk,\val)$, and $(\Lpk,\val,\des)$. All of these have been addressed by Darij Grinberg. Specifically, Grinberg proved that $\Epk$ is shuffle-compatible using a $P$-partition argument \cite{Grinberg}, noted that $(\Pk,\udr)$ and $(\Lpk,\val)$ are both equivalent to $\Epk$, and found counterexamples showing that $(\Pk,\val,\des)$ and $(\Lpk,\val,\des)$ are not shuffle-compatible \cite{Grinberg2017}.

Prior to this, Grinberg had shown that $\QSym$ is a ``dendriform algebra'' \cite{Grinberg2017a}, an algebra whose multiplication can
be split into a ``left multiplication'' and a ``right multiplication'' satisfying certain nice axioms. 
Together with the shuffle-compatibility of $\Epk$, Grinberg
proved that ${\cal A}_{\Epk}$ is a dendriform quotient of $\QSym$. More generally, he proved that a descent statistic is a dendriform quotient of $\QSym$ if and only if it is
both ``left-shuffle-compatible'' and  ``right-shuffle-compatible'', which are combinatorial conditions that, together, refine the notion of shuffle-compatibility. Other descent statistics that Grinberg has shown to be both left- and right-shuffle-compatible include the descent number $\des$, the pair $(\des,\maj)$, and the left peak set $\Lpk$. On the other hand, the major index $\maj$, the peak set $\Pk$, and the right peak set $\Rpk$ are neither left- nor right-shuffle-compatible.

From Theorem \ref{t-pairfalse}, we know that a pair of two 
shuffle-compatible statistics need not be shuffle-compatible.
Hence, we pose the following question.
\begin{question}
Suppose that $\st_{1}$ and $\st_{2}$ are shuffle-compatible statistics.
Are there simple conditions that imply that the pair $(\st_{1},\st_{2})$ is shuffle-compatible?
\end{question}
Similarly, if a pair is shuffle-compatible, then that does not imply
that the individual statistics in the pair are both shuffle-compatible.
\begin{question}
Suppose that the pair $(\st_{1},\st_{2})$ is shuffle-compatible. Are there simple conditions that imply that $\st_{1}$ and $\st_{2}$ are both shuffle-compatible?
\end{question}
Recall that Goulden \cite{Goulden1985} and Stadler \cite{Stadler1999} gave combinatorial proofs for the shuffle-compatibility of $(\des,\maj)$, and in Section \ref{s-bijproof} we provided combinatorial proofs for the shuffle-compatibility of the descent set $\Des$ and partial descent sets $\Des_{i,j}$.

\begin{question}
Can we find combinatorial proofs for the shuffle-compatibility of other statistics?
\end{question}
Finally, we present the following conjecture.
\begin{conjecture}
Every shuffle-compatible permutation statistic is a descent statistic.
\end{conjecture}

\vspace{10bp}

\noindent \textbf{Acknowledgements.} We thank Bruce Sagan and an anonymous referee for providing extensive feedback on a preliminary version of this paper, as well as Marcelo Aguiar, Sami Assaf, Darij Grinberg, and Kyle Petersen for helpful discussions on this project.

\appendix

\section{Tables of permutation statistics}

The following table summarizes every permutation statistic $\st$
that we know to be shuffle-compatible, along with their shuffle algebra
${\cal A}_{\st}$ and the dimension of the $n$th homogeneous component
of ${\cal A}_{\st}$.
\goodbreak
\renewcommand{\arraystretch}{1.4}
\noindent \begin{center}
\begin{longtable}{|>{\centering}p{2.1in}|>{\centering}p{1.85in}|>{\centering}p{1.85in}|}
\caption{Shuffle-compatible permutation statistics}
\tabularnewline
\hline 
Permutation Statistic & Shuffle Algebra & Dimension of $n$th Homogeneous Component\tabularnewline
\hline 
$\Des$ & $\QSym$ & $2^{n-1}$\tabularnewline
\hline 
$\des$ & Theorem \ref{t-dessc} & $n$\tabularnewline
\hline 
$\maj$, $\comaj$ & Theorem \ref{t-majsc} & ${n \choose 2}+1$\tabularnewline
\hline 
$(\des,\maj)$, $(\des,\comaj)$ & Theorem \ref{t-descomajsc} & ${n \choose 3}+n$\tabularnewline
\hline 
$\Pk,\Val$ & Algebra of peaks $\Pi$ & $F_{n}$\tabularnewline
\hline 
$\pk,\val,\epk$ & Theorem \ref{t-pksc} & $\left\lfloor (n+1)/2\right\rfloor $\tabularnewline
\hline
$\Lpk,\Rpk$ & Algebra of left peaks $\Pi^{(\ell)}$ & $F_{n+1}$\tabularnewline
\hline 
$\lpk,\rpk,\lr$ & Theorem \ref{t-lpksc} & $\left\lfloor n/2\right\rfloor +1$\tabularnewline
\hline
$\Des_{1,0},\Des_{0,1},\sir,\lir,\sfr,\lfr$ &  & $2$ \tabularnewline
\hline  
$\Des_{i,j}$ &  & $2^{i+j}$ (if $i+j \leq n-1$) \tabularnewline
\hline  
$(\pk,\des)$, $(\val,\des)$, $(\epk,\des)$ & Theorem \ref{t-pkdessc} & $\left\lfloor (n+1)^{2}/4\right\rfloor $\tabularnewline
\hline 
$(\lpk,\des), (\rpk,\des), (\lr,\des)$ & Theorem \ref{t-lpkdessc} & $\left\lfloor n^{2}/4\right\rfloor +1$\tabularnewline
\hline 
$\udr$, $(\lpk,\val)$, $(\lpk,\pk)$, $(\lpk,\epk)$, $(\rpk,\val)$, $(\rpk,\pk)$, $(\rpk,\epk)$, $(\lr,\val)$, $(\lr,\pk)$, $(\lr,\epk)$ & Theorem \ref{t-udrsc} & $n$\tabularnewline
\hline
$(\udr,\des)$, $(\lpk,\val,\des)$, $(\lpk,\epk,\des)$, $(\lpk,\pk,\des)$ & Theorem \ref{t-udsc} & $\binom{n}{2}+1$\tabularnewline
\hline  
$\Epk$ & \cite{Grinberg} & $F_{n+2}-1$ \tabularnewline
\hline  
\end{longtable}
\par\end{center}

The next table gives a partial list of equivalences, $r$-equivalences, $c$-equivalences, and $rc$-equivalences among permutation statistics that are studied in this paper. Not all of these are explicitly proven in this paper, but the proofs are very straightforward. 
We leave out some redundancies such as $\sir \sim_{c} \lir$---omitted since we include $\sir \sim \lir$---as well as equivalences like $(\Lpk,\val,\des) \sim (\Lpk,\br,\des)$, which is an immediate consequence of  $(\Lpk,\val) \sim (\Lpk,\br)$.

\begin{center}
\begin{longtable}{|>{\centering}p{3.5in}|>{\centering}p{0.2in}|c|}
\caption{Equivalences among permutation statistics}
\tabularnewline
\cline{1-1} \cline{3-3} 
Equivalences &  & $r$-Equivalences\tabularnewline
\cline{1-1} \cline{3-3} 
$\Des\sim\Lpk\cup\Val\sim(\Lpk,\Val)$ &  & $\Lpk\sim_{r}\Rpk$\tabularnewline
\cline{1-1} \cline{3-3} 
$\val\sim\epk$ &  & $\lpk\sim_{r}\rpk$\tabularnewline
\cline{1-1} \cline{3-3} 
$\rpk\sim\epk$ &  & $\sir\sim_{r}\lfr$\tabularnewline
\cline{1-1} \cline{3-3} 
$\rpk\sim\lr$ &  & $\sfr\sim_{r}\lir$\tabularnewline
\cline{1-1} \cline{3-3} 
$\udr\sim(\lpk,\val)$ & \multicolumn{1}{>{\centering}p{0.2in}}{} & \multicolumn{1}{c}{}\tabularnewline
\cline{1-1} \cline{3-3} 
\multirow{2}{3.5in}{\hphantom{i}$\Epk\sim(\Epk,\val)\sim(\Epk,\udr)\sim(\Epk,\br)$\vspace{1bp}
\hphantom{aaaaaaaaaa}$\sim(\Lpk,\val)\sim(\Lpk,\udr)\sim(\Pk,\udr)$} &  & $c$-Equivalences\tabularnewline
\cline{3-3} 
 &  & $\Pk\sim_{c}\Val$\tabularnewline
\cline{1-1} \cline{3-3} 
$\sir\sim\lir\sim\Des_{1,0}$ &  & $\pk\sim_{c}\val$\tabularnewline
\cline{1-1} \cline{3-3} 
$\sfr\sim\lfr\sim\Des_{0,1}$ & \multicolumn{1}{>{\centering}p{0.2in}}{} & \multicolumn{1}{c}{}\tabularnewline
\cline{1-1} \cline{3-3} 
$(\Pk,\val)\sim(\Pk,\br)$ &  & $rc$-Equivalences\tabularnewline
\cline{1-1} \cline{3-3} 
$(\Lpk,\val)\sim(\Lpk,\br)$ &  & $(\pk,\des)\sim_{rc}(\val,\des)$\tabularnewline
\cline{1-1} \cline{3-3} 
$(\pk,\val)\sim(\pk,\br)\sim(\val,\br)$ &  & $(\lpk,\val)\sim_{rc}(\lpk,\pk)$\tabularnewline
\cline{1-1} \cline{3-3} 
\end{longtable}

\par\end{center}

\phantomsection

\bibliographystyle{amsplain}
\addcontentsline{toc}{section}{\refname}\bibliography{bibliography}

\end{document}